\documentclass{jcmlatex}
\def\JCMvol{xx}
\def\JCMno{x}
\def\JCMyear{200x}
\def\JCMreceived{xxx}
\def\JCMrevised{xxx}
\def\JCMaccepted{xxx}

\usepackage{amsmath}
\usepackage{mathrsfs}
\usepackage{subcaption}
\usepackage{graphicx}
\usepackage{color}

\newcommand{\bu}{{\bf u}}

\newcommand{\bw}{{\bf w}}

\newcommand{\bx}{{\bf x}}

\newcommand{\be}{{\bf e}}
\newcommand{\bv}{{\mathbf v}}

\def\bbf{{\bf f}}

\def\bt{{\bf t}}
\def\bn{{\bf n}}

\def\bw{{\bf w}}

\def\3bar{{|\hspace{-.02in}|\hspace{-.02in}|}}

\begin{document}

\markboth{J.~JIA, L.~YANG AND Q. ZHAI}{The pressure-robust WG method for Stokes-Darcy problem}

\title{THE PRESSURE-ROBUST WEAK GALERKIN FINITE ELEMENT METHOD FOR STOKES-DARCY PROBLEM}
\author{Jiwei Jia, Lin Yang \and Qilong Zhai\footnote{Corresponding author}
\thanks{Department of Mathematics, Jilin University, Changchun, China. \\ 
	Email: jiajiwei@jlu.edu.cn, ~linyang22@mails.jlu.edu.cn, ~zhaiql@jlu.edu.cn}}

\maketitle

\begin{abstract}
In this paper, we propose a pressure-robust weak Galerkin (WG) finite element scheme to solve the Stokes-Darcy problem. To construct the pressure-robust numerical scheme, we use the divergence-free velocity reconstruction operator to modify the test function on the right side of the numerical scheme.  We prove the error between the velocity function and its numerical solution is independent of the pressure function and viscosity coefficient. Moreover, the errors of the velocity function and the pressure function reach the optimal convergence orders under the energy norm,  as validated by both theoretical analysis and numerical results.
\end{abstract}



\section{Introduction}
This paper considers the Stokes-Darcy model, 
which couples the Stokes equations in the free flow region with the Darcy equations in the porous medium region. 

In the free flow region $\Omega_s$, the flow is governed by the following Stokes equations:
\begin{eqnarray}
	-\nabla \cdot T(\bu_s,p_s)=\bbf_s,~~ {\text{in}} ~\Omega_s,\label{StokesEquation1}\\
	\nabla \cdot \bu_s =g_s,~~ {\text{in}} ~ \Omega_s,\label{StokesEquation2}\\
	\bu_s={\bf{0}},~~ {\text{on}} ~ \Gamma_s,\label{StokesEquation3}
\end{eqnarray}
where $T(\bu_s,p_s)=2\mu D(\bu_s)-p_s I$ is the symmetric stress tensor and $D(\bu_s)= \color{black}{(\nabla u_s + ( \nabla u_s)^T)/2}$ is the strain tensor. {\color{black}Define $\mu$ and $I$ as the viscosity coefficient and the identity tensor, respectively.}

In the porous medium region $\Omega_d$, the flow is governed by the Darcy equations in the mixed formulation:
\begin{eqnarray}
	\mu \kappa^{-1}\bu_d +\nabla p_d =\bbf_d,~~ \text{in} ~ \Omega_d,\label{DarcyEquation1}\\
	\nabla \cdot \bu_d =g_d,~~ \text{in} ~ \Omega_d,\label{DarcyEquation2}\\
	\bu_d \cdot \bn_d = 0,~~ \text{on} ~ \Gamma_d,\label{DarcyEquation3}
\end{eqnarray}
where $\kappa$ is permeability tensor. {\color{black}Denote by $g=(g_s, g_d)$ the source satisfying}
$$
\int_{\Omega_s} g_s  ~{\rm{d}} \Omega_s + \int_{\Omega_d} g_d ~{\rm{d}}\Omega_d =0.
$$
This model is coupled with three interface conditions: the mass conservation condition, the balance of forces, and the Beavers-Joseph-Saffman (BJS) condition.
\begin{eqnarray}
	\bu_s \cdot \bn_s =\bu_d \cdot \bn_s,~~ \text{on} ~\Gamma_{sd},\label{InterfaceCondition1}\\
p_s-{\color{black}2\mu  D(\bu_s)\bn_s \cdot \bn_s} =p_d,~~ \text{on} ~\Gamma_{sd},\label{InterfaceCondition2}\\
	{\color{black}-2D(\bu_s)\bn_s \cdot \bt_s} =\alpha \kappa^{-\frac{1}{2}}\bu_s \cdot \bt_s,~~ \text{on} ~\Gamma_{sd}.\label{InterfaceCondition3}
\end{eqnarray}
Here $\bn_s$ and $\bt_s$ are the unit outward normal vector and the unit tangent vector on the interface $\Gamma_{sd}$, respectively. {\color{black}And} $\alpha$ is an empirical parameter obtained through experiments. Let $\Omega = \Omega_s \cup \Omega_d$ be an open bounded domain in $\mathbb{R}^2$. This model is shown in Figure \ref{figure1}.

 This model simulates the transport of pollutants from rivers into aquifers in environmental science. Moreover, it finds many applications in hydrology, biofluid dynamics and other fields. So far, various numerical methods have been proposed to numerically solve the Stokes-Darcy problem, including the mixed finite element method \cite{ MFEMSD2,MFEMSD1, MFEMSD3,MFEMPR1}, the finite volume method \cite{FVMSD}, the discontinuous Galerkin finite element method \cite{ DGSD2, DGSD3, DGSD4,DGSD}, the virtual element method \cite{VEMSD}, the weak Galerkin finite element method  \cite{WGStokesDrcyWang, WGSD1, WGSD2, WGStokesDrcyPeng2, WGStokesDrcyPeng1}, etc. 

 When using the piecewise linear polynomial combined with the lowest-order Raviart-Thomas (RT) element to discrete the velocity function in the Stokes equations and the Darcy equations, the error of the velocity function is independent of viscosity coefficient and the pressure function in \cite{MFEMPR1}. We call this property of the error as pressure robustness. However, most numerical methods are not pressure-robust. Consequently, if viscosity coefficient is very small or the approximation of the pressure function is inaccurate, the approximation of the velocity function is correspondingly compromised. This non-pressure robustness also occurs when the Stokes equations are solved numerically. To obtain the pressure-robust numerical results, scholars have proposed various methods: the mixed finite element method based on an exact de Rham complex \cite{girault2012finite, PressureRobust2017}, grad-div stabilization \cite{ Graddiv1988, Graddiv2009, Graddiv2009-2, Graddiv2002, Graddiv2004}, appropriate reconstructions of test functions \cite{Reconstruction1, Reconstruction2, Reconstruction3}, etc.


The WG method was first proposed in \cite{wang2013weak} for solving second-order elliptic equation. In contrast to the finite element method, the WG method employs the discontinuous weak function space as the approximate function space and replaces the classical differential operators with the weak differential operators in the variational formulation. In \cite{WGStokes1}, the authors develop a stabilized WG method for the Stokes equations. The errors of velocity function are dependent on the pressure function and viscosity coefficient. To obtain the pressure-robust WG scheme, the authors proposed to construct velocity reconstruction operator with lowest-order divergence-free $CW_0$ element in \cite{PressureRobust2020}. To obtain higher-order accuracy, the definition of the $H(div)$ element on polygonal meshes is introduced in \cite{PressureRobust2021}. Subsequently, arbitrary order WG schemes on polygonal meshes are proposed.

In this paper, we propose an unified pressure-robust WG scheme for the Stokes-Darcy problem. The WG finite element is employed to discrete the velocity functions and pressure functions in the Stokes equations and Darcy equations. We use the $H(div)$ element to construct a velocity reconstruction operator. This operator is employed to modify the test function on the right-hand side, ultimately yielding the pressure-robust WG scheme. We obtain the error between the velocity function and its numerical solution is independent of both the pressure function and viscosity coefficient. And the error of the pressure function only depends on the velocity function and viscosity coefficient. Moreover, the numerical solutions converge optimally to the exact solutions in theoretical analysis and numerical examples.

This paper is mainly divided into the following parts. In Section 2, we give the definitions of the WG spaces and velocity reconstruction operator, and the WG scheme is proposed. Section 3 is devoted to establishing the stability of the numerical scheme. In Section 4, we present the error equations for the velocity and pressure functions in both the free flow region and porous medium region. Subsequently, we derive error estimates in the $H^1$ norm in Section 5. In Section 6, we provide two numerical examples to validate the efficiency of the pressure-robust WG scheme.

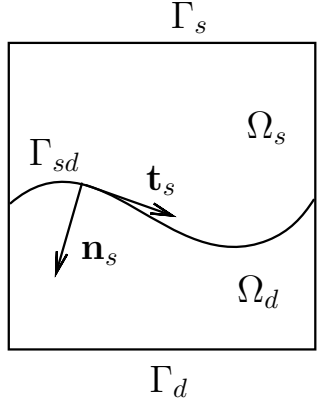
\begin{figure}
	\centering
	\setlength{\unitlength}{1bp}%
	\begin{picture}(125.67, 162.35)(0,0)
		\put(0,0){\includegraphics{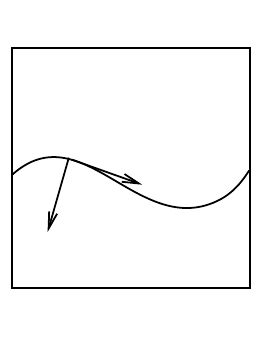}}
		\put(56.87,84.66){\fontsize{14.23}{17.07}\selectfont $\bt_s$}
		\put(32.60,58.61){\fontsize{14.23}{17.07}\selectfont $\bn_s$}
		\put(94.40,104.08){\fontsize{14.23}{17.07}\selectfont $\Omega_s$}
		\put(90.86,42.28){\fontsize{14.23}{17.07}\selectfont $\Omega_d$}
		\put(66.14,145.57){\fontsize{14.23}{17.07}\selectfont $\Gamma_{s}$}
		\put(58.20,8.73){\fontsize{14.23}{17.07}\selectfont $\Gamma_{d}$}
		\put(13.17,93.93){\fontsize{14.23}{17.07}\selectfont $\Gamma_{sd}$}
	\end{picture}%
	\caption{The Stokes-Darcy Model}
	\label{figure1}
\end{figure}

\section{The Weak Galerkin Finite Element Scheme}
In this section, we introduce the definitions of some WG spaces and weak differential operators. Next, based on the definition of the velocity reconstruction operator, the pressure-robust WG scheme is presented. 

\subsection{Weak Galerkin finite element spaces}
Denote by $\mathcal{T}_{s,h}$ and $\mathcal{T}_{d,h}$ the simplicial partitions in the free flow region $\Omega_s$ and porous medium region $\Omega_d$, respectively, and set $\mathcal{T}_h=\mathcal{T}_{s,h} \cup \mathcal{T}_{d,h}$. For $T \in \mathcal{T}_h$, define the area and the diameter of $T$ as $|T|$ and $h_T$, respectively. Let $h=\max_{T \in \mathcal{T}_h} h_T$ be the mesh size. Let $\mathcal{E}_{s,h}$ and $\mathcal{E}_{d,h}$ be the set of all edges in the Stokes domain $\Omega_s$ and Darcy domain $\Omega_d$, respectively. And set $\mathcal{E}_h=\mathcal{E}_{s,h} \cup \mathcal{E}_{d,h}$. Denote by $\mathcal{E}_{sd,h}$ all edges on the interface $\Gamma_{sd}$.

Define the weak function as $\bv=\{\bv_0, \bv_b\}$, where $\bv_0$ and $\bv_b$ are interior function and boundary function, respectively. It should be noted that $\bv_b$ has one unique value on the edges of $\mathcal{E}_h \setminus \mathcal{E}_{sd,h}$ and has two values $\bv_{s,b}$ and $\bv_{d,b}$ on the edge of $\mathcal{E}_{sd,h}$. And there's no relationship between $\bv_0$ and $\bv_b$. Let $V_{s,h}$ and $V_{d,h}$ be the WG spaces  in the free flow region $\Omega_s$ and porous medium region $\Omega_d$, respectively. Set $V_h=V_{s,h} \cup V_{d,h}$.  Define the WG space for pressure function in the domain $\Omega$ as $W_h$. For a given integer $k \geqslant 1$,
\begin{align*}
	V_{s,h}=&\{\bv_{s,h}=\{\bv_{s,0},\bv_{s,b}\}, \bv_{s,0}|_{T} \in [P_k(T)]^2 \,\text{ in }\, T \in \mathcal{T}_{s,h},\bv_{s,b}|_e \in [P_{k}(e)]^2 \,\text{ on} \, e \in \mathcal{E}_{s,h}
	\},\\
	V_{d,h}=&\{\bv_{d,h}=\{\bv_{d,0},\bv_{d,b}\}, \bv_{d,0}|_{T} \in [P_k(T)]^2 \, \text{in} \, T \in \mathcal{T}_{d,h},\bv_{d,b}|_e = v_{d,b}\bn_e, \, v_{d,b} \in P_{k}(e)\\
	&\, \text{on} \, e \in \mathcal{E}_{d,h} \setminus \mathcal{E}_{sd,h}, \bv_{d,b} |_e \in [P_{k}(e)]^2 \, \text{on} \, e \in \mathcal{E}_{sd,h}
	\},\\ 
	V_h^0=&\{ \bv_h \in V_h, \bv_{s,b}=\mathbf{0} \, \text{on} \, e \in \mathcal{E}_{s,h} \cap \partial \Omega, \bv_{d,b} \cdot \bn_e =0 \,\text{on} \,  e\in \mathcal{E}_{d,h} \cap \partial \Omega, \bv_{s,b}=\bv_{d,b} \, \text{on} \, e \in \mathcal{E}_{sd,h} \},\\
	W_h=&\{q: q \in L_0^2(\Omega), q|_T \in P_{k}(T)  \,\text{in} \, T \in \mathcal{T}_h\},		
\end{align*}
where $\bn_e$ is the unit outward normal vector on the edge $e$. {\color{black}Denote by $P_{k}(T)$ } the space of polynomials on $T$ with degree no more than $k$. {\color{black}And} $P_k(e)$ represents the the space of polynomials on $e$ with degree no more than $k$.

Then, we define some weak differential operators. 
\begin{definition}\cite{WGStokes1}
	For any $\bv \in V_{s,h}$ and $T \in \mathcal{T}_{s,h}$, {\color{black} its discrete weak gradient operator $\nabla_w \bv \in [P_{k-1}(T)]^{2\times 2}$ satisfies the following equation}: 
	\begin{equation}
		(\nabla_w \bv, \varphi)_T=-(\bv_0,\nabla \cdot \varphi)_T + \langle \bv_b,\varphi \cdot \bn \rangle_{\partial T},~\forall ~\varphi \in [P_{k-1}(T)]^{2\times 2}, \label{weakgradient}
	\end{equation}
	where $\bn$ is the unit outward normal vector on $\partial T$.
\end{definition}

Then, define the discrete weak strain tensor as $D_w(\bv)=(\nabla_w \bv+(\nabla_w \bv)^T)/2$. Similarly, we give the definition of the discrete weak divergence operator.

\begin{definition}\cite{WGStokes1}
	For each $\bv \in V_h$ and $T \in \mathcal{T}_{h}$, {\color{black} its discrete weak divergence operator $\nabla_w \cdot \bv \in P_{k}(T)$ satisfies the following equation:} 
	\begin{equation}
		(\nabla_w \cdot \bv,\tau)_T=-(\bv_0, \nabla \tau)_T + {\color{black}\langle \bv_b \cdot \bn, \tau  \rangle_{\partial T}} , ~ \forall ~ \tau \in P_{k}(T),\label{weakdivergence}
	\end{equation}
	where $\bn$ is the unit outward normal vector on $\partial T$.
\end{definition}

Next, we introduce the following Sobolev spaces\cite{PressureRobust2021.2}:
\begin{eqnarray*}
	H(div, \Omega) &=& \{\bv \in [L^2(\Omega)]^2,\, \nabla \cdot \bv \in  L^2(\Omega)
	\},\\
	H_0(div, \Omega) &=& \{\bv \in H(div, \Omega),\, \bv \cdot \bn|_{\partial \Omega}=0
	\},\\
	RT_{k}(T) &=& [P_{k}(T)]^2 + \bx P_{k}(T), \, \bx \in \mathbb{R}^2,  T \in \mathcal{T}_h.
\end{eqnarray*}
Then we give the definition of the $H(div)$-conforming velocity reconstruction operator $R_T$.
\begin{definition}\label{RT operator}\cite{PressureRobustCDG2021}
	Let $R_T: {\color{black}V_h|_T \rightarrow RT_{k}(T)}$ be the velocity reconstruction operator. For all $\bv_h =\{\bv_0, \bv_b\} \in V_h$, the operator $R_T$ satisfies:
	\begin{eqnarray}
		\int_T R_T(\bv_h) \cdot \bw ~dT& =& \int_T \bv_0 \cdot \bw ~dT,~~\forall \,\bw \in [P_{k-1}(T)]^2, \label{RTEquation1}\\
		\int_e (R_T(\bv_h) \cdot \bn) \phi ~ds &=& \int_e (\bv_b \cdot \bn) \phi ~ds,~~\forall \, \phi \in P_{k}(e),\label{RTEquation2}
	\end{eqnarray}
	where $T \in \mathcal{T}_h$ and $e \in \mathcal{E}_h$.
\end{definition}

\subsection{The numerical scheme}
For $T \in \mathcal{T}_h$, denote by $Q_0$ the $L^2$ projection operator from $[L^2(T)]^2$ onto $[P_k(T)]^2$. For $e \in \mathcal{E}_h$, {\color{black}define $Q_b$ as} the $L^2$ projection operator from $[L^2(e)]^2$ onto $[P_{k}(e)]^2$. Set $Q_h= \{Q_0, Q_b\}$. Next we give some bilinear forms as follows:
\begin{eqnarray*}
	\begin{split}
		s(\bv,\bw)=& \sum_{T \in \mathcal{T}_{s,h}}h_T^{-1}  \langle \bv_{s,0} -\bv_{s,b},\bw_{s,0} -\bw_{s,b} \rangle_{\partial T}\\
		&+\sum_{T \in \mathcal{T}_{d,h}}h_T^{-1}  \langle (\bv_{d,0} -\bv_{d,b}) \cdot \bn_d, (\bw_{s,0} -\bw_{s,b}) \cdot \bn_d \rangle_{\partial T}\\
		a(\bv,\bw)=&\sum_{T \in \mathcal{T}_{s,h}}(2\mu D_w(\bv_{s,h}),D_w(\bw_{s,h}))_T+\sum_{T \in \mathcal{T}_{d,h}}(\mu \kappa^{-1} \bv_{d,0},\bw_{d,0})_T,\\
		a_s(\bv,\bw)=&a(\bv,\bw)+\mu s(\bv,\bw)+\sum_{e \in \mathcal{E}_{sd,h}}  \langle \alpha \mu \kappa^{-\frac{1}{2}} \bu_{s,b} \cdot \bt_s, \bv_{s,b} \cdot \bt_s\rangle_e,\\
		b(\bv,p)=&-\sum_{T \in \mathcal{T}_h}(\nabla_w \cdot \bv,p)_T,
	\end{split}
\end{eqnarray*}
for any $\bv$, $\bw \in V_h^0$.

Based on these bilinear forms, we present the pressure-robust WG scheme in Algorithm \ref{Algorithm1}.
\begin{algorithm}\label{Algorithm1}
	For the Stokes-Darcy problem (\ref{StokesEquation1})-(\ref{InterfaceCondition3}), the weak Galerkin finite element scheme is to find {\color{black}$\mathbf{u}_h \in V_h^0$}, $\bu_{s,b} |_{\Gamma_s}={\bf{0}}$, $\bu_{d,b}|_{\Gamma_d} \cdot \bn_d =0$ and $p_h \in W_h$ to satisfy
	\begin{eqnarray}
		a_s(\bu_h,\bv_h)+b(\bv_h,p_h)&=&(\bbf_s,R_T(\bv_{s,h}))+(\bbf_d,R_T(\bv_{d,h})),\label{WGscheme1}\\
		b(\bu_h,q_h)&=&-(g, q_h) ,\label{WGscheme2}
	\end{eqnarray}
	for all $\bv_h \in V_h^0$ and $q_h \in W_h$.
\end{algorithm}

 Moreover, we introduce the standard WG scheme presented as Algorithm \ref{Algorithm2} for comparison.
\begin{algorithm}\label{Algorithm2}
	For the Stokes-Darcy problem (\ref{StokesEquation1})-(\ref{InterfaceCondition3}), the standard weak Galerkin finite element scheme is to find {\color{black}$\mathbf{u}_h \in V_h^0$},$\bu_{s,b} |_{\Gamma_s}={\bf{0}}$, $\bu_{d,b}|_{\Gamma_d} \cdot \bn_d =0$  and $p_h \in W_h$ to satisfy
	\begin{eqnarray}
		a_s(\bu_h,\bv_h)+b(\bv_h,p_h)&=&(\bbf_s, \bv_{s,0})+(\bbf_d, \bv_{d,0}),\label{WGscheme3}\\
		b(\bu_h,q_h)&=&-(g, q_h),\label{WGscheme4}
	\end{eqnarray}
	for all $\bv_h \in V_h^0$ and $q_h \in W_h$.
\end{algorithm}

\section{Stability and Well-posedness}
In this section, we establish the stability of the WG scheme.
Let $\mathcal{Q}_h$ and $\mathbb{Q}_h$ be the $L^2$ projection operators onto $P_{k}(T)$ and $[P_{k-1}(T)]^{2 \times 2}$, respectively.
\begin{lemma}
	For $\bu \in [H^1(\Omega)]^2$ and $T \in \mathcal{T}_h$, we have the following properties:
	\begin{eqnarray}
		(\nabla_w(Q_h \bu),\varphi)_T&=&(\mathbb{Q}_h(\nabla \bu),\varphi)_T, ~ \forall \, \varphi \in [P_{k-1}(T)]^{2 \times 2},\label{weak gradient exchange 1}\\
		(\nabla_w \cdot (Q_h \bu), \tau)_T&=&(\mathcal{Q}_h(\nabla \cdot \bu),\tau)_T, \, \forall \, \tau \in P_{k}(T).\label{weak divergence exchange 1}
	\end{eqnarray}
\end{lemma}
The proof refers to Lemma 4.2 in \cite{WGStokes1}.

Similarly, for $\varphi \in [P_{k-1}(T)]^{2 \times 2}$, we get
\begin{eqnarray*}
	\begin{split}
		(\nabla_w^T Q_h \bu, \varphi)_T =&(\nabla_w Q_h \bu , \varphi^T)_T\\
		=&-(Q_0 \bu , \nabla \cdot \varphi^T)_T + \langle Q_b \bu ,\varphi^T \bn \rangle_{\partial T}\\
		=&-(\bu, \nabla \cdot \varphi^T)_T+\langle \bu, \varphi^T \bn \rangle_{\partial T}\\
		=&(\nabla \bu, \varphi^T)_T\\
		=&(\mathbb{Q}_h((\nabla \bu)^T), \varphi)_T.
	\end{split}
\end{eqnarray*}
So the discrete weak strain tensor has the same property as the weak gradient operator:
\begin{eqnarray}
	(D_w(Q_h \bu), \varphi)_T=(\mathbb{Q}_h (D(\bu)), \varphi)_T, \quad \forall \, \varphi \in [P_{k-1}(T)]^{2 \times 2}.\label{weakgradientexchange2}
\end{eqnarray} 
Now, we define a semi-norm in the WG space $V_h$ as follows:
\begin{equation}
	\begin{split}
		\3bar\bv\3bar^2=&\sum_{T \in \mathcal{T}_{s,h}}\left(\| D_w(\bv_{s,h})\|_T^2 +\frac{1}{2} h_T^{-1} \|\bv_{s,0}-\bv_{s,b}\|^2_{\partial T}\right)\\
		&+\sum_{T \in \mathcal{T}_{d,h}} \left( \frac{1}{2} \| \kappa^{-\frac{1}{2}} \bv_{d,0} \|_T^2 + \frac{1}{2} h_T^{-1} \| (\bv_{d,0} - \bv_{d,b}) \cdot \bn_e\|^2_{\partial T}\right)\\
		&+\frac{\alpha}{2}\sum_{e \in \mathcal{E}_{sd,h}} \|\kappa^{-\frac{1}{4}} \bv_{s,b} \cdot \bt_s\|_e^2.
	\end{split}
\end{equation}

\begin{lemma}
	$\3bar \cdot \3bar$ provides a norm in $V_h^0$.
\end{lemma}
\begin{proof}
     Set $\3bar \bv_h \3bar =0 $ for some $\bv_h \in V_h^0$, according to the definition of $\3bar \cdot \3bar$, we obtain
	\begin{align}
		D_w(\bv_{s,h})=0, \, \text{in} ~ T \in \mathcal{T}_{s,h};
		\bv_{s,0}=\bv_{s,b}, \, \text{on} ~e \in \mathcal{E}_{s,h} \cup \mathcal{E}_{sd,h},\label{norm1}\\
		\bv_{d,0}=0, \, \text{in} ~T \in \mathcal{T}_{d,h};
		(\bv_{d,0} -\bv_{d,b}) \cdot \bn_e =0, \, \text{on} ~e \in \mathcal{E}_{d,h} \cup \mathcal{E}_{sd,h},\label{norm2}\\
		\bv_{s,b} \cdot \bt_s =0, \, \text{on} ~e \in \mathcal{E}_{sd,h}\label{norm3}.
	\end{align}
	
	Based on Eq.(\ref{norm2}), we get $\bv_{d,b}=\bf{0}$ $\text{on} ~ e \in \mathcal{E}_{d,h}$ and $\bv_{d,b} \cdot \bn_e =0$ on $ e \in \mathcal{E}_{sd,h}$. Combining Eq.(\ref{norm3}) with $\bv_{s,b} = \bv_{d,b}$ on $e \in \mathcal{E}_{sd,h}$, we get  $\bv_{s,b} = \bv_{d,b}=0$ on $ e \in \mathcal{E}_{sd,h}$.	
	

	Using the discrete $Korn's$ inequality, Eqs.(\ref{norm1})- (\ref{norm3}) and Lemma 4.1 in  \cite{WGStokesDrcyPeng1}, we get $\nabla \bv_{s,0} =0 $ on $T \in \mathcal{T}_{s,h}$. Therefore, $\bv_{s,0} $ is a constant on $T \in \mathcal{T}_{s,h}$. Combing $\bv_{s,0} = \bv_{s,b}$ on $e \in \mathcal{E}_{s,h} \cap \mathcal{E}_{sd,h}$ with  $\bv_{s,b}=0$ on $e \in \mathcal{T}_s \cap \mathcal{E}_{sd,h}$, we derive $\bv_{s,0} =0$ on $T \in \mathcal{T}_{s,h}$. And it follows from $\bv_{d,0}=0 $ and  $\bv_{d,b}=0$ that $\bv_h =0 $. So $\3bar \cdot \3bar$ is a norm on $V_h^0$. \hfill$\square$
\end{proof}

\begin{lemma}\label{Abounded}
	For any $\bv,\bw \in V_h^0$, we have
	\begin{eqnarray*}
		|a(\bv,\bw)|&\leqslant& \3bar \bv \3bar \cdot \3bar \bw \3bar.
	\end{eqnarray*}
\end{lemma}

\begin{lemma}(Inf-Sup Condition)\label{InfSupCondition}
	There is a positive constant $C$ such that
	\begin{eqnarray}\label{InfSupinequlity}
		\sup_{\bv \in V_h^0}\frac{b(\bv,q)}{\3bar \bv \3bar} \geqslant C \| q\|, \, q \in W_h,
	\end{eqnarray}
	where $C$ is independent of the mesh size $h$.
\end{lemma}
\begin{proof}
	According to \cite{bookfiniteelementmethod1,bookfiniteelementmethod2,bookfiniteelementmethod3,bookfiniteelementmethod5,bookfiniteelementmethod4,WGStokes1}, it is clear that for $q \in W_h$, there is a $\bv \in [H_0^1(\Omega)]^2$ satisfying $- \nabla \cdot \bv = q$ and the following inequality holds true:
	\begin{eqnarray*}
		\| \bv \|_{1,\Omega} \leqslant C \|q\|_{0,\Omega}.
	\end{eqnarray*}
	Now we take $\tilde{\bv} = Q_h \bv$ to prove $\3bar \tilde{\bv} \3bar \leqslant \| \bv \|_{1,\Omega}$. To prove this inequality, we estimate every term in $\3bar \tilde{\bv} \3bar$.
	
	For the first term, according to  Eq.(\ref{weak gradient exchange 1}) and the property of $L^2$ projection operator, we have
	\begin{eqnarray*}
		\| D_w (\tilde{\bv})\|^2_{\Omega_s}=\sum_{T \in \mathcal{T}_{s,h}} \|D_w(Q_h \bv)\|_T^2 \leqslant \sum_{T \in \mathcal{T}_{s,h}} \|\nabla_w Q_h \bv\|_T ^2 = \sum_{T \in \mathcal{T}_{s,h}} \|\mathbb{Q}_h(\nabla \bv)\|_T^2 \leqslant C \|\bv\|^2_{1,\Omega_s}.
	\end{eqnarray*}
	
	For the second term, based on the triangle inequality, the trace inequality, and the projection inequality, we get
	\begin{eqnarray*}
		\begin{split}
			&\sum_{T \in \mathcal{T}_{s,h}} h_T^{-1} \|\tilde{\bv}_{s,0} - \tilde{\bv}_{s,b}\|_{\partial T}^2 \\
			=&\sum_{T \in \mathcal{T}_{s,h}} h_T^{-1} \| Q_0 \bv -Q_b \bv\|_{\partial T}^2\\
			\leqslant & 2 \sum_{T \in \mathcal{T}_{s,h}} h_T^{-1} \|Q_0 \bv-\bv\|_{\partial T}^2  +h_T^{-1} \|Q_b \bv-\bv\|_{\partial T}^2 \\
			\leqslant & C \|\bv\|^2_{1,\Omega_s}. 
		\end{split}
	\end{eqnarray*}
	
	Similarly, for the third term, we obtain 
	\begin{eqnarray*}
		\|\kappa^{-\frac{1}{2}} Q_0 \bv\|_{\Omega_d}^2 =\sum_{T \in \mathcal{T}_{d,h}} \|\kappa^{-\frac{1}{2}} Q_0 \bv \|_T^2 \leqslant C \sum_{T \in \mathcal{T}_{d,h}} \|\bv\|_T^2  \leqslant C \|\bv\|_{1,\Omega_d}^2.
	\end{eqnarray*}
	
	For the fourth term, it follows from the triangle inequality, the trace inequality and the projection inequality that 
	\begin{eqnarray*}
		\begin{split}
			&\sum_{T \in \mathcal{T}_{d,h}} h_T^{-1} \| (Q_0 \bv -((Q_b \bv) \cdot \bn_e)\bn_e)\cdot \bn_e\|_{\partial T}^2\\
			\leqslant & C \sum_{T \in \mathcal{T}_{d,h}} h_T^{-1} \|Q_0 \bv - Q_b \bv \|_{\partial T}^2\\
			\leqslant & C\|\bv\|_{1,\Omega_d}^2.
		\end{split}
	\end{eqnarray*}
	
	For the fifth term, using the property of $L^2$ projection operator, the trace inequality and the Poincar$\acute{e}$ inequality in $\Omega$, we have
	\begin{eqnarray*}
		\|\kappa^{-\frac{1}{4}} Q_b \bv \cdot \bt_s \|_{\mathcal{E}_{sd,h}}^2 \leqslant C \| Q_b \bv \|_{\mathcal{E}_{sd,h}} ^2 \leqslant C \| \bv \|_{\mathcal{E}_{sd,h}}^2 \leqslant C (h^{-1} \|\bv\|_{\Omega_d}^2+h \|\nabla \bv\|_{\Omega_d}^2 )
		\leqslant C \|\bv\|_{1,\Omega}^2.
	\end{eqnarray*}
	To sum up, $\3bar \tilde{\bv} \3bar$ $\leqslant$ $C  \|\bv\|_{1,\Omega}$.
	Moreover, by Eq.(\ref{weak divergence exchange 1}), we get
	\begin{eqnarray*}
		(\nabla_w \cdot \tilde{\bv}, q)= (\nabla_w \cdot Q_h \bv, q) =(\mathcal{Q}_h (\nabla \cdot \bv), q)=(\nabla \cdot \bv, q)= -\|q\|_{0,\Omega}^2.
	\end{eqnarray*}
	Therefore, the estimate (\ref{InfSupinequlity}) holds true.
	\hfill$\square$
\end{proof}

\begin{lemma}\cite{WGStokesDrcyPeng1}
	For any $\bv_h \in V_h$, we have
	\begin{eqnarray}\label{gradientenergy}
		\sum\limits_{T \in \mathcal{T}_{s,h}} \|\nabla \bv_{s,0}\|_T^2 \leqslant \3bar \bv_h \3bar^2.
	\end{eqnarray}
\end{lemma}

Based on the above lemmas, we obtain the following existence and uniqueness of the WG scheme in the Algorithm \ref{Algorithm1}.
\begin{theorem}
	The weak Galerkin finite element scheme (\ref{WGscheme1})-(\ref{WGscheme2}) has a unique solution.
\end{theorem}

\section{Error Analysis}In this section, we first give the properties of the operator $R_T$. Based on these properties, we derive the error equations of the velocity function $\bu$ and the pressure function $p$. 
\subsection{The properties of the velocity reconstruction operator}
For the operator $R_T$, we have the following properties.
\begin{lemma}
	For any $\bv_h \in V_h$, we have 
	\begin{equation}
		\nabla \cdot (R_T(\bv_h))=\nabla_w \cdot \bv_h,\label{RTproperty1}
	\end{equation}
	and
	\begin{equation}
		\sum_{T \in \mathcal{T}_h}\|R_T(\bv_h)-\bv_0\|_T^2 \leqslant C h^2 \3bar \bv_h \3bar^2.\label{RTproperty2}
	\end{equation}
\end{lemma}
\begin{proof}
	For $T \in \mathcal{T}_h$ and $q \in P_{k}(T)$, according to integration by parts, the definition of the weak divergence operator and Eqs.(\ref{RTEquation1})-(\ref{RTEquation2}), we have
	\begin{eqnarray*}
		\begin{split}
			(\nabla \cdot (R_T(\bv_h)), q)_T =&-(R_T(\bv_h), \nabla q)_T+ \langle R_T(\bv_h) \cdot \bn, q\rangle_{\partial T}\\
			=&-(\bv_0, \nabla q)_T+ \langle \bv_b, q\bn \rangle_{\partial T}\\
			=&(\nabla_w \cdot \bv_h, q)_T
		\end{split}
	\end{eqnarray*}
	Taking $q=\nabla \cdot  (R_T(\bv_h))-\nabla_w \cdot \bv_h$ to obtain Eq.(\ref{RTproperty1}).
	
	Next for $T \in \mathcal{T}_h$, denote by $\Pi_T$ and $\Pi_T^{k-1}$ the $L^2$ projection operators from $[L^2(T)]^2$ onto $RT_{k}(T)$ and $[P_{k-1}(T)]^2$, respectively. Since $\bv_0 \in [P_k(T)]^2 \subset RT_{k}(T)$, we have $\Pi_T(\bv_0)=\bv_0$. Therefore, by the example 12.6 in \cite{ern_finite_2021}, we obtain
	\begin{eqnarray*}
		\begin{split}
			&\|R_T(\bv_h)-\bv_0\|_T^2=\|R_T(\bv_h)-\Pi_T(\bv_0)\|_T^2\\
			\leqslant& \|\Pi_T^{k-1}(R_T(\bv_h)-\Pi_T(\bv_0))\|_T^2+h_T \sum_{e \in \partial T} \|(R_T(\bv_h)-\Pi_T(\bv_0)) \cdot \bn_e \|_e^2.
		\end{split}
	\end{eqnarray*}
For the first term, according to the property of projection operator and Eq.(\ref{RTEquation1}), we have
\begin{eqnarray*}
	\begin{split}
		&\|\Pi_T^{k-1}(R_T(\bv_h)-\Pi_T(\bv_0))\|_T^2\\
		=&(\Pi_T^{k-1}(R_T(\bv_h)-\bv_0),\Pi_T^{k-1}(R_T(\bv_h)-\bv_0))_T\\
		=&(\Pi_T^{k-1}(R_T(\bv_h)-\bv_0),R_T(\bv_h)-\bv_0)_T\\
		=&0.
	\end{split}
\end{eqnarray*}
For the second term, by the Eqs.(\ref{RTEquation1})-(\ref{RTEquation2}), we get
\begin{eqnarray*}
	\begin{split}
		&\|(R_T(\bv_h)-\Pi_T(\bv_0)) \cdot \bn_e\|_e^2\\
		=&\langle (R_T(\bv_h)-\bv_0) \cdot \bn_e, ((R_T(\bv_h)-\bv_0) \cdot \bn_e \rangle_e\\
		=&\langle \bv_0 \cdot \bn_e, \bv_0 \cdot \bn_e \rangle_e - 2\langle \bv_0 \cdot \bn_e, \bv_b \cdot \bn_e\rangle_e+ \langle \bv_b \cdot \bn_e, \bv_b \cdot \bn_e \rangle_e\\
		=&\|(\bv_0 - \bv_b)\cdot \bn_e \|_e^2.
	\end{split}
\end{eqnarray*}	
To sum up, we have
\begin{eqnarray*}
	\sum_{T \in \mathcal{T}_{h}} \|R_T(\bv_h)-\bv_0\|_T^2 \leqslant C \sum_{T \in \mathcal{T}_{h}} h_T \|(\bv_0 - \bv_b)\cdot \bn_e \|_{\partial T}^2 \leqslant C h^2 \3bar \bv_h \3bar^2.
\end{eqnarray*}	
The proof of the estimate (\ref{RTproperty2}) is complete.
\hfill$\square$

\end{proof}

\subsection{Error equation}
Denote by $\{\bu_h , p_h\}$ the numerical solutions of the WG scheme (\ref{WGscheme1})-(\ref{WGscheme2}). Define $\{\bu_i , p_i\}$ with $i=s,d$ as the solutions of the problem (\ref{StokesEquation1})-(\ref{InterfaceCondition3}). 
The errors of $\bu ~\text{and}\ p$ are defined by
\begin{eqnarray}
	e_h=Q_h \bu - \bu_h, ~ \varepsilon_h=\mathcal{Q}_h p-p_h.
\end{eqnarray}

\begin{lemma}\label{EE}
	For the $(\bu_i, p_i)\in [H^1(\Omega_i)]^2 \times L^2(\Omega_i)$ with $i=s,d$ satisfying the problem (\ref{StokesEquation1})-(\ref{InterfaceCondition3}), we derive
	\begin{eqnarray}\label{uprojectionEquation1}
		\begin{split}
			a(Q_h \bu, \bv)+b(\bv, \mathcal{Q}_h p)
			=&(\bbf_s,R_T(\bv_{s,h}))+(\bbf_d,R_T(\bv_{d,h}))+\varphi_{\bu}(\bv), {\color{black}\forall ~\bv \in V_h^0}
		\end{split}
	\end{eqnarray}
	where
	\begin{eqnarray}
		\varphi_{\bu}(\bv)&=&-\ell_1(\bu_s,\bv_{s,h})-\ell_2(\bu_d,\bv_{d,h})+\ell_3(\bu_s,\bv_{s,h})-\ell_4(\bu_s,\bv_{s,h}),\\
		\ell_1(\bu_s,\bv_{s,h})&=&\sum_{T \in \mathcal{T}_{s,h}}(\nabla \cdot (2 \mu D(\bu_s)),\bv_{s,0}-R_T(\bv_{s,h}))_T,\\
		\ell_2(\bu_d,\bv_{d,h})&=&\sum_{T \in \mathcal{T}_{d,h}}({\color{black}\mu \kappa^{-1}  \bu_d}, R_T(\bv_{d,h})-\bv_{d,0})_T,\\
		\ell_3(\bu_s,\bv_{s,h})&=&\sum_{T \in \mathcal{T}_{s,h}} \langle 
		\bv_{s,0}-\bv_{s,b},2\mu D(\bu_s)\cdot \bn_s-2\mu \mathbb{Q}_h(D(\bu_s))\cdot \bn_s
		\rangle_{\partial T_s},\\
		\ell_4(\bu_s,\bv_{s,h})&=&\sum_{e \in \mathcal{E}_{sd,h}} \langle
		(2 \mu D(\bu_s) \cdot \bn_s) \bn_s, R_T(\bv_{s,h}) \cdot \bn_s - \bv_{s,b} \cdot \bn_s
		\rangle_e.
	\end{eqnarray}
\end{lemma}

\begin{proof}
	Multiply on two sides of Eq.(\ref{StokesEquation1}) by  $R_T(\bv_{s,h})$ and integrate to obtain
	\begin{eqnarray*}
		(\bbf_s,R_T(\bv_{s,h}))_{\Omega_s}&=&(-\nabla \cdot T(\bu_s,p_s),R_T(\bv_{s,h}))_{\Omega_s}\\
		&=&-(\nabla \cdot T(\bu_s,p_s),\bv_{s,0})_{\Omega_s}+(\nabla \cdot T(\bu_s,p_s),\bv_{s,0}-R_T(\bv_{s,h}))_{\Omega_s}\\
		&=&I_1 + I_2.
	\end{eqnarray*} 	
 For $I_1$, according to integration by parts and the definition of $T(\bu_s,p_s)$, we have
\begin{eqnarray}
	\begin{split}\label{VSEquation1}
		I_1=&\sum_{T \in \mathcal{T}_{s,h}}(-\nabla \cdot T(\bu_s, p_s),\bv_{s,0})_T\\
		=&\sum_{T \in \mathcal{T}_{s,h}}(T(\bu_s,p_s),\nabla \bv_{s,0})_T-\langle T(\bu_s,p_s)\bn_s,\bv_{s,0}
		\rangle_{\partial T}\\
		=&\sum_{T \in \mathcal{T}_{s,h}}(2\mu D(\bu_s),\nabla \bv_{s,0})_T -(\nabla \cdot \bv_{s,0},p_s)_T-\langle 2\mu D(\bu_s) \bn_s, \bv_{s,0}-\bv_{s,b} \rangle_{\partial T}\\
		&+\langle p_s \bn_s ,\bv_{s,0} \rangle_{\partial T}
		-\langle 2\mu D(\bu_s)\bn_s, \bv_{s,b} \rangle_{\mathcal{E}_{sd,h}},
	\end{split}	
\end{eqnarray}
where we have used the fact that {\color{black} $\sum_{e \in \mathcal{E}_{s,h}\setminus \mathcal{E}_{sd,h} } \langle 2 \mu D(\bu_s) \bn_s, \bv_{s,b} \rangle_e =0$}.

Next, it follows from the property of $L^2$ projection operator, integration by parts, the definition of weak gradient operator and Eq.(\ref{weakgradientexchange2}) that
\begin{eqnarray}
	\begin{split}
		&(2\mu D(\bu_s),\nabla \bv_{s,0})_T\\
		=&(2\mu \mathbb{Q}_h(D(\bu_s)),\nabla \bv_{s,0})_T\\
		=&-(\bv_{s,0},\nabla \cdot (2\mu \mathbb{Q}_h(D(\bu_s))))_T
		+\langle 
		\bv_{s,0},2\mu \mathbb{Q}_h(D(\bu_s))  \bn_s
		\rangle_{\partial T}\\
		=&(\nabla_w \bv_{s,h},2\mu \mathbb{Q}_h(D(\bu_s)))_T+\langle 
		\bv_{s,0}-\bv_{s,b},2\mu \mathbb{Q}_h(D(\bu_s)) \bn_s \rangle_{\partial T}\\
		=&(\nabla_w \bv_{s,h},2\mu D_w(Q_h \bu_s))_T+
		\langle 
		\bv_{s,0}-\bv_{s,b},2\mu \mathbb{Q}_h(D(\bu_s))  \bn_s \rangle_{\partial T}.
	\end{split}
\end{eqnarray}

Hence, {\color{black}the first term }$I_1$ can be rewritten as
\begin{eqnarray}
	\begin{split}
		I_1=&\sum_{T \in \mathcal{T}_{s,h}}
		(\nabla_w \bv_{s,h},2\mu D_w(Q_h \bu_s))_T-\ell_3(\bu_s, \bv_{s,h})\\
		&-(\nabla \cdot \bv_{s,0},p_s)_T 
		-\langle 2 \mu D(\bu_s) \bn_s, \bv_{s,b}\rangle_{\mathcal{E}_{sd,h}}+\langle
		p_s \bn_s, \bv_{s,0}
		\rangle_{\partial T}.
	\end{split}
\end{eqnarray}
As to $I_2$, using integration by parts and Eqs.(\ref{RTEquation1})-(\ref{RTEquation2}) that
\begin{eqnarray}
	\begin{split}
		I_2=&(\nabla \cdot T(\bu_s,p_s),\bv_{s,0}-R_T(\bv_{s,h}))_{\Omega_s}\\
		=&\sum_{T \in \mathcal{T}_{s,h}} (\nabla \cdot (2 \mu D(\bu_s)),\bv_{s,0}-R_T(\bv_{s,h}))_T-(\nabla p_s,\bv_{s,0}-R_T(\bv_{s,h}))_T\\
		=&\ell_1(\bu_s,\bv_{s,h})+\sum_{T \in \mathcal{T}_{s,h}} (\nabla \cdot \bv_{s,0},p_s)_T
		-\langle p_s, \bv_{s,0} \cdot \bn_s \rangle_{\partial T}\\
		&+\left(\sum_{T \in
			\mathcal{T}_{s,h}} -(\nabla \cdot R_T(\bv_{s,h}),p_s)_T+\langle R_T(\bv_{s,h}) \cdot \bn_s, p_s \rangle_{\partial T} \right)\\
		=&\ell_1(\bu_s,\bv_{s,h})+\sum_{T \in \mathcal{T}_{s,h}} (\nabla \cdot \bv_{s,0},p_s)_T
		-\langle p_s, \bv_{s,0} \cdot \bn_s \rangle_{\partial T}\\
		&-\sum_{T \in \mathcal{T}_{s,h}} (\nabla_w \cdot \bv_{s,h},p_s)_T+\sum_{T \in \mathcal{T}_{s,h}} \langle  {\color{black}\bv_{s,b} \cdot \bn_s}, p_s  \rangle_{\partial T}.
	\end{split}
\end{eqnarray}
Then by the interface conditions (\ref{InterfaceCondition1})-(\ref{InterfaceCondition3}), we have
\begin{eqnarray}\label{VSEquation2}
	\begin{split}
		I_1+I_2=&\sum_{T \in \mathcal{T}_{s,h}} (2 \mu D_w(Q_h \bu_s),D_w \bv_{s,h})_T-\sum_{T \in \mathcal{T}_{s,h}}(\nabla_w \cdot \bv_{s,h},\pi_h p_s)_T\\
		&+\langle p_d, R_T(\bv_{s,h}) \cdot \bn_s \rangle_{\mathcal{E}_{sd,h}}
		+\langle
		\alpha \mu \kappa^{-\frac{1}{2}}\bu_s \cdot \bt_s, \bv_{s,b} \cdot \bt_s
		\rangle_{\mathcal{E}_{sd,h}}\\
		=&{\color{black}(\bbf_s,R_T(\bv_{s,h}))_{\Omega_s}-\ell_1(\bu_s,\bv_{s,h})+\ell_3(\bu_s,\bv_{s,h})}.
	\end{split}
\end{eqnarray}
In the porous medium region, multiply on two sides of Eq.(\ref{DarcyEquation1}) by $R_T(\bv_{d,h})$  and integrate to get
\begin{eqnarray*}
	\begin{split}
		&(\bbf_d,R_T(\bv_{d,h}))_{\Omega_d}\\
		=&(\mu \kappa^{-1} \bu_d,R_T(\bv_{d,h}))_{\Omega_d}+(\nabla p_d, R_T(\bv_{d,h}))_{\Omega_d}\\
		=&\sum_{T \in \mathcal{T}_{d,h}}(\mu \kappa^{-1} \bu_d,\bv_{d,0})_{T}+\sum_{T \in \mathcal{T}_{d,h}}(\mu \kappa^{-1} \bu_d,R_T(\bv_{d,h})-\bv_{d,0})_T+\sum_{T \in \mathcal{T}_{d,h}}(\nabla p_d, R_T(\bv_{d,h}))_T.
	\end{split}
\end{eqnarray*}
According to the property of $L^2$ projection operator, we have
\begin{eqnarray*}
	(\mu \kappa^{-1} \bu_d,\bv_{d,0})_T=(\mu \kappa^{-1} Q_0 \bu_d,\bv_{d,0})_T,\, \forall \, T \in \mathcal{T}_{d,h}.
\end{eqnarray*}
It follows from integration by parts, Eqs.(\ref{RTEquation1})-(\ref{RTEquation2}), and Eq.(\ref{RTproperty1}) that
\begin{eqnarray*}
	\begin{split}
		&\sum_{T \in \mathcal{T}_{d,h}}(\nabla p_d,R_T(\bv_{d,h}))_{T}\\
		=&\sum_{T \in \mathcal{T}_{d,h}} \left(
		-(\nabla \cdot R_T(\bv_{d,h}),p_d)_T+\langle
		R_T(\bv_{d,h}) \cdot \bn_d, p_d\rangle_{\partial T}
		\right)\\
		=&\sum_{T \in \mathcal{T}_{d,h}} \left(
		-(\nabla_w \cdot \bv_{d,h},{\color{black}\mathcal{Q}_h} p_d)_T+\langle R_T(\bv_{d,h}) \cdot \bn_d,p_d \rangle_{\partial T}
		\right).			
	\end{split}
\end{eqnarray*}
Therefore, we get
\begin{eqnarray}\label{VDEquation1}
	\begin{split}
		&\sum_{T \in \mathcal{T}_{d,h}}\left(  
		(\mu \kappa^{-1} Q_0 \bu_d,\bv_{d,0})_T-(\nabla_w \cdot \bv_{d,h},{\color{black}\mathcal{Q}_h} p_d)_T 
		\right)+\sum_{e \in \mathcal{E}_{sd,h}} \langle {\color{black}\bv_{d,h}\cdot \bn_d}, p_d \rangle_e \\
		=&(\bbf_d, R_T(\bv_{d,h}))_{\Omega_d}-\ell_2(\bu_d,\bv_{d,h}).
	\end{split}
\end{eqnarray}
Adding Eq.(\ref{VSEquation2}) to Eq.(\ref{VDEquation1}) yields
\begin{eqnarray*}
	\begin{split}
		&\sum_{T \in \mathcal{T}_{s,h}} (2 \mu D_w(Q_h \bu_s),D_w \bv_{s,h})_T-\sum_{T \in \mathcal{T}_{s,h}}(\nabla_w \cdot \bv_{s,h},{\color{black}\mathcal{Q}_h} p_s)_T\\
		-&\sum_{T \in \mathcal{T}_{d,h}}(\nabla_w \cdot \bv_{d,h},{\color{black}\mathcal{Q}_h} p_d)_T +\sum_{T \in \mathcal{T}_{d,h}}(\mu \kappa^{-1} Q_0 \bu_d,\bv_{d,0})_T
		+\langle
		\alpha \mu \kappa^{-\frac{1}{2}}\bu_s \cdot \bt_s, \bv_{s,b} \cdot \bt_s
		\rangle_{\Gamma_{sd}}\\
		=&(\bbf_s,R_T(\bv_{s,h}))_{\Omega_s}+(\bbf_d, R_T(\bv_{d,h}))_{\Omega_d}+\varphi_{\bu}(\bv).
	\end{split}
\end{eqnarray*}
The proof of Eq.(\ref{uprojectionEquation1}) is complete.\hfill$\square$
\end{proof}

\begin{lemma}\label{error equation}
For any $\bv \in V_h^0$ and $q \in W_h$, we have the following error equations:
\begin{eqnarray}
	a_s(\be_h,\bv)+b(\bv,\varepsilon_h)&=&\varphi_{\bu}(\bv)+\mu s(Q_h \bu ,\bv),\label{ErrorEquation1}\\
	b(\be_h,q)&=&{\color{black}0}, \label{ErrorEquation2}
\end{eqnarray}
where $\varphi_{\bu}(\bv)$ is the same as the definitions of Lemma \ref{EE}.
\end{lemma}
\begin{proof}
Since the solutions $(\bu, p)$ satisfy the problem (\ref{StokesEquation1})-(\ref{InterfaceCondition3}), according to Lemma \ref{EE}, we have
\begin{align*}
	a(Q_h \bu, \bv)+b(\bv, \mathcal{Q}_h p)
	=(\bbf_s,R_T(\bv_{s,h}))_{\Omega_s}+(\bbf_d, R_T(\bv_{d,h}))_{\Omega_d}+\varphi_{\bu}(\bv).
\end{align*}
Adding $\mu s(Q_h \bu, \bv)$ to two sides of the above equation and subtracting (\ref{WGscheme1}) yields Eq.(\ref{ErrorEquation1}). Next, by Eq.(\ref{weak divergence exchange 1}), for any $q \in W_h$, we get 
\begin{eqnarray}\label{eep6}
	\begin{split}
		-(\nabla_w \cdot (Q_h \bu), q)=&\sum_{T \in \mathcal{T}_{h}} -(\nabla_w \cdot (Q_h \bu), q)_T\\
		=&\sum_{T \in \mathcal{T}_{h}} -(\mathcal{Q}_h(\nabla \cdot \bu) ,q)_T\\
		=&\sum_{T \in \mathcal{T}_{h}} -(\nabla \cdot \bu, q)_T\\
		=&-(g, q).
	\end{split}
\end{eqnarray}
Then, subtracting Eq.(\ref{eep6}) from Eq.(\ref{WGscheme2}), we derive
\begin{eqnarray}
	b(\be_h,q)=0.
\end{eqnarray}
Hence, the proof of the error equations is complete.\hfill$\square$
\end{proof}

\section{Error Estimates in the Energy Norm}
In this section, we obtain the optimal estimates for error  $\be_h$ of the velocity function and the error $\varepsilon_h$ of the pressure function.

\begin{lemma}\label{H1 estimates}
Suppose $\bu_i \in[H^{k+1}(\Omega_i)]^2 $ with $i=s,d$, we have
\begin{eqnarray}
	|\ell_1(\bu_s,\bv_{s,h})|  & \leqslant &\mu C h^k\|\bu_s\|_{k+1, \Omega_s} \3bar \bv_h \3bar,\label{H1 estimates 1}\\
	|\ell_2(\bu_d,\bv_{d,h})| &\leqslant & \mu C h^k \|\bu_d\|_{k+1, \Omega_d}\3bar \bv_h \3bar,\label{H1 estimates 2}\\
	|\ell_3(\bu_s,\bv_{s,h})| &\leqslant & \mu C h^k\|\bu_s\|_{k+1, \Omega_s} \3bar \bv_h \3bar,\label{H1 estimates 3}\\
	|\ell_4(\bu_s,\bv_{s,h})| & \leqslant& \mu Ch^k\|\bu_s\|_{k+1, \Omega_s} \3bar \bv_{h}\3bar,\label{H1 estimates 4}\\
	|s(Q_h \bu ,\bv)| &\leqslant & C h^k(\|\bu_s\|_{k+1, \Omega_s}+\| \bu_d \|_{k+1, \Omega_d}) \3bar \bv_h \3bar,\label{H1 estimates 5}
\end{eqnarray}
where $\bv_{s,h}=\{\bv_{s,0}, \bv_{s,b}\} \in  V_{s,h}$ and $\bv_{d,h}=\{\bv_{d,0}, \bv_{d,b}\} \in  V_{d,h}$.
\end{lemma}

\begin{proof}
For the estimate (\ref{H1 estimates 1}), according to the Cauchy-Schwarz inequality, Eq.(\ref{RTEquation1}), and the projection inequality, we get 
\begin{eqnarray}
	\begin{split}
		|\ell_1(\bu_s, \bv_{s,h})|=&\left|\sum_{T \in \mathcal{T}_{s,h}} (\nabla \cdot (2 \mu D(\bu_s)),R_T(\bv_{s,h})-\bv_{s,0})_T\right|\\
		=&2\mu \left|\sum_{T \in \mathcal{T}_{s,h}} (\nabla \cdot D(\bu_s)-\Pi_h^{k-1}(\nabla \cdot D(\bu_s)),R_T(\bv_{s,h})-\bv_{s,0})_T\right|\\
		\leqslant & 2\mu \left( 
		\sum_{T \in \mathcal{T}_{s,h}} \| \nabla \cdot D(\bu_s)- \Pi_h^{k-1}(\nabla \cdot D(\bu_s))\|_T^2
		\right)^{\frac{1}{2}}\\
		&\left(
		\sum_{T \in \mathcal{T}_{s,h}} \| R_T(\bv_{s,h})-\bv_{s,0}\|_T^2
		\right)^{\frac{1}{2}}\\
		\leqslant&\mu C h^k \|\bu_s\|_{k+1,\Omega_s} \3bar \bv_h \3bar,
	\end{split}
\end{eqnarray}
where we have used the fact that $\sum_{T \in \mathcal{T}_{s,h}} \left(\Pi_h^{k-1}(\nabla \cdot D(\bu_s)), R_T(\bv_{s,h})-\bv_{s,0} \right)_T=0$.

Similarly, we have
\begin{eqnarray}
	\begin{split}
		|\ell_2(\bu_d,\bv_{d,h})|=&\left|\sum_{T \in \mathcal{T}_{d,h}}(\mu \kappa^{-1}\bu_d ,R_T(\bv_{d,h})-\bv_{d,0})_T\right|\\
		=&\left|\sum_{T \in \mathcal{T}_{d,h}}\mu \kappa^{-1}(\bu_d-\Pi_h^{k-1}\bu_d ,R_T(\bv_{d,h})-\bv_{d,0})_T\right|\\
		\leqslant &\mu C \left( 
		\sum_{T \in \mathcal{T}_{d,h}} \| \bu_d- \Pi_h^{k-1} \bu_d\|_T^2
		\right)^{\frac{1}{2}}
		\left(
		\sum_{T \in \mathcal{T}_{d,h}} \| R_T(\bv_{d,h})-\bv_{d,0}\|_T^2
		\right)^{\frac{1}{2}}\\
		\leqslant &  \mu C h^k \|\bu_d\|_{k+1,\Omega_d} \3bar \bv_h\3bar.
	\end{split}
\end{eqnarray}
As to the estimate (\ref{H1 estimates 3}), based on the Cauchy-Schwarz inequality, the trace inequality and the projection inequality, we obtain
\begin{eqnarray}\label{HEP1}
	\begin{split}
		|\ell_3(\bu,\bv)|=&\left|\sum_{T \in \mathcal{T}_{s,h}} \langle \bv_{s,0} -\bv_{s,b}, 2\mu D(\bu_s) \cdot \bn_s - 2\mu \mathbb{Q}_h (D(\bu_s))\cdot \bn_s \rangle_{\partial T}\right|\\
		\leqslant& 2 \mu \sum_{T \in \mathcal{T}_{s,h}} \| \bv_{s,0} -\bv_{s,b}  \|_{\partial T} \| D(\bu_s)  - \mathbb{Q}_h (D(\bu_s))\|_{\partial T}\\
		\leqslant& \mu C \left( \sum_{T \in \mathcal{T}_{s,h}}  \| \bv_{s,0} -\bv_{s,b}  \|^2_{\partial T} \right)^{\frac{1}{2}} \left( \sum_{T \in \mathcal{T}_{s,h}} \| D(\bu_s)- \mathbb{Q}_h(D(\bu_s)) \|^2_{\partial T}  \right)^{\frac{1}{2}}\\
		\leqslant & \mu Ch^{k} \|\bu_s\|_{k+1,\Omega_s} \3bar \bv_h \3bar.
	\end{split}
\end{eqnarray}
In a similar way, we derive
\begin{eqnarray}\label{HEP5}
	\begin{split}
		&|s(Q_h \bu,\bv_h)|\\
		=&\Big|\sum_{T \in \mathcal{T}_{s,h}} h_T^{-1} \langle Q_0 \bu_s- Q_b \bu_s
		,\bv_{s,0} -\bv_{s,b} \rangle_{\partial T} \\
		&+\sum_{T \in \mathcal{T}_{d,h}} h_T^{-1} \langle
		(Q_0 \bu_d - ((Q_b \bu_d)\cdot \bn_e)\bn_e) \cdot \bn_e,(\bv_{d,0} - \bv_{d,b}) \cdot \bn_e
		\rangle_{\partial T}\Big|\\
		\leqslant & \sum_{T \in \mathcal{T}_{s,h}} h_T^{-1} \|Q_0 \bu_s -\bu_s\|_{\partial T}
		\|\bv_{s,0} -\bv_{s,b}\|_{\partial T} \\
		&+\sum_{T \in \mathcal{T}_{d,h}} h_T^{-1} \| Q_0 \bu_d - Q_b \bu_d\|_{\partial T} \|(\bv_{d,0} -\bv_{d,b}) \cdot \bn_e \|_{\partial T}\\
		\leqslant & \left( \sum_{T \in \mathcal{T}_{s,h}} h_T^{-1} \|Q_0 \bu_s - \bu_s\|_{\partial T}^2+\sum_{T \in \mathcal{T}_{d,h}} h_T^{-1} \|Q_0 \bu_d -Q_b \bu_d\|_{\partial T}^2
		\right)^{\frac{1}{2}} \\
		&\left(
		\sum_{T \in \mathcal{T}_{s,h}} h_T^{-1} \|\bv_{s,0}-\bv_{s,b}\|_{\partial T}^2+\sum_{T \in \mathcal{T}_{d,h}} h_T^{-1} \| (\bv_{d,0} -\bv_{d,b}) \cdot \bn_e\|_{\partial T}
		\right)^{\frac{1}{2}}\\
		\leqslant & Ch^k (\|\bu_s\|_{k+1,\Omega_s}+\|\bu_d\|_{k+1,\Omega_d})\3bar  \bv_h \3bar.
	\end{split}
\end{eqnarray}
The proof of the above lemma is complete.\hfill$\square$
\end{proof}

\begin{theorem}
Let $\bu_i\in [H^{k+1}(\Omega_i)]^2$ with $i=s,d$ satisfy problem (\ref{StokesEquation1})-(\ref{InterfaceCondition3}). Then the errors $\be_h$ and $\varepsilon_h$ have the following estimates:
\begin{eqnarray}
	\3bar \be_h \3bar \leqslant C h^k \|\bu\|_{k+1},\label{ErrorE1}\\
	\|\varepsilon_h\| \leqslant \mu C h^k \|\bu\|_{k+1}.\label{ErrorE2}
\end{eqnarray}
\end{theorem}
\begin{proof}
Taking $\bv= \be_h$ in Eq.(\ref{ErrorEquation1}) and $q= \varepsilon_h$ in Eq.(\ref{ErrorEquation2}) leads to
\begin{align}
	     a_s(\be_h,\be_h)+b(\be_h,\varepsilon_h)
		=&\varphi_{\bu}(\be_h)+\mu s(Q_h \bu ,\be_h),\label{ErrorE3}\\
		b(\be_h,\varepsilon_h)=&0.\label{ErrorE4}
\end{align}
Substituting Eq.(\ref{ErrorE4}) to Eq.(\ref{ErrorE3}) yields
\begin{eqnarray*}
	a_s(\be_h,\be_h)=\varphi_{\bu}(\be_h)+\mu s(Q_h \bu ,\be_h).
\end{eqnarray*}
By Lemma \ref{H1 estimates}, we have $2\mu \3bar \be_h \3bar^2 \leqslant \mu C h^k \|\bu\|_{k+1} \3bar \be_{h} \3bar$. Therefore, the estimate (\ref{ErrorE1}) holds true. 

Next, for $\varepsilon_h \in W_h$, according to the inf-sup condition (\ref{InfSupinequlity}) and error equations (\ref{ErrorEquation1})-(\ref{ErrorEquation2}),  we have 
\begin{eqnarray*}
	\begin{split}
		C\|\varepsilon_h\| \leqslant& \sup\limits_{\bv \in V_h} \frac{b(\bv, \varepsilon_h)}{\3bar \bv \3bar}\\ \leqslant& \left|\sup\limits_{\bv \in V_h}
		\frac{a_s(\be_h, \bv)-\varphi_{\bu}(\bv)-
			\mu s(Q_h \bu ,\bv)}{\3bar \bv \3bar}\right| \\
		\leqslant& \mu Ch^k \| \bu \|_{k+1}
	\end{split}
\end{eqnarray*}
Hence the proof of the estimation (\ref{ErrorE2}) is complete.\hfill$\square$
\end{proof}

\section{Numerical Results} In this section, we give two numerical examples to validate the efficiency of the proposed WG scheme. First, we take a fixed coefficient $ \mu $ and different mesh size to demonstrate the orders of convergence. Next, we fix the mesh size and take different coefficients $ \mu $ to verify that the error of $\bu$ is independent of the viscosity coefficient $ \mu $. Finally, we present a numerical example to confirm the errors of $\bu$ and $p$ are independent of the pressure function $p$.

\begin{example}\label{example 1}
	In this example, we consider the Stokes-Darcy problem on the Stokes domain $\Omega_s =(0,\pi) \times (0, \pi)$ and Darcy domain $\Omega_d =(0, \pi) \times (-\pi, 0)$. The interface is described as $\Gamma =(0, \pi) \times \{0\}$. The exact solutions on the Stokes domain and Darcy domain are as follows:
\begin{eqnarray*}
\bu_s=\left(\begin{array}{ccc}
	2 \sin y \cos y \cos x \\
	(\sin^2 y -2) \sin x 
\end{array}\right),\,
p_s=\sin x \sin y,
\end{eqnarray*}
and
\begin{eqnarray*}
\bu_d=\left(\begin{array}{ccc}
	-(e^y-e^{-y}) \cos x\\
	-(e^y+e^{-y})\sin x 
\end{array}\right),\,
p_d=(e^y-e^{-y})\sin x.
\end{eqnarray*}
\end{example}

In this example, we use the uniform triangular meshes which constructed as follows: the domain is uniformly partitioned into $n \times n$ rectangles, and each rectangular element is divided by a diagonal line with a positive slope. For this example, we compare the accuracy of Algorithm \ref{Algorithm1} and Algorithm \ref{Algorithm2}. We use $P_1$ and $P_2$ WG elements to solve the problem (\ref{StokesEquation1})-(\ref{InterfaceCondition3}). The errors and convergence orders of Algorithm \ref{Algorithm1} and Algorithm \ref{Algorithm2} are reported in Tables \ref{table30}-\ref{table9}. The errors of the velocity function $\bu$ and the pressure function $p$ achieve the optimal convergence orders, which are validated in the algorithms with different WG spaces and viscosity coefficients. It's easy to see from Table \ref{table30} that the errors of the velocity function $\bu$ deteriorate when the viscosity coefficient $\mu$ is too small in the Algorithm \ref{Algorithm2}. However, as predicted by the theory for Algorithm \ref{Algorithm1}, when the viscosity coefficient $\mu$ is either too large or too small, the error of the velocity function $\bu$ remain a constant, and the error of the pressure function is proportional to the viscosity coefficient.

Secondly, we present the results obtained by selecting a fixed mesh size $n=16$ and varying viscosity coefficients $\mu$ in Figures \ref{figure6.2} and \ref{figure6.3}. It can be observed that the errors of the velocity function in Algorithm \ref{Algorithm1} remain unchanged regardless of the viscosity coefficient, indicating that Algorithm \ref{Algorithm1} is pressure-robust. However, Algorithm \ref{Algorithm2} does not exhibit pressure robustness. The results presented in the two figures illustrate that the errors of the velocity function are influenced by the viscosity coefficient. And when the viscosity coefficient is small, the error deteriorates. As evident from the comparison, Algorithm \ref{Algorithm1} can improve the approximation of the velocity function and the pressure function.

\begin{example}\label{example 2}
	In this example, we consider the Stokes-Darcy problem on the Stokes domain $\Omega_s =(0,\frac{1}{2}) \times (0, 1)$ and Darcy domain $\Omega_d =(\frac{1}{2}, 1) \times (0, 1)$. The interface is described as $\Gamma =(0, 1) \times \{\frac{1}{2}\}$. The exact solutions in the free flow region and the porous medium region are as follows:
\begin{eqnarray*}
\bu_s=\left(\begin{array}{ccc}
	0 \\
	0 
\end{array}\right),\,
p_s=(xy)^3-\frac{1}{16},
\end{eqnarray*}
and
\begin{eqnarray*}
\bu_d=\left(\begin{array}{ccc}
	0\\
	0 
\end{array}\right),\,
p_d=(xy)^3-\frac{1}{16}.
\end{eqnarray*}
\end{example}

In this example, we use the same triangular meshes as in Example \ref{example 1} and set the viscosity coefficient $\mu =1$. The velocity functions in both the free flow region and porous medium region are zero. The numerical results of Algorithms \ref{Algorithm1} and \ref{Algorithm2} are presented in Tables \ref{table13}-\ref{table16}. For the velocity function and the pressure function, the errors obtained by Algorithm \ref{Algorithm1} are almost zero, indicating good approximation of the exact solutions. These results also show that the error of the velocity function is independent of the pressure function. However, the errors obtained by Algorithm \ref{Algorithm2} are not zero, but only optimally converge to the exact solutions. By comparison, Algorithm \ref{Algorithm1} performs better in approximating the exact solutions than Algorithm \ref{Algorithm2}.

\begin{table}
	\caption{The errors of Algorithm \ref{Algorithm2} in Example \ref{example 1} with different $\mu$}
	\label{table30}
	\centering
    \resizebox{0.9\textwidth}{!}{
		\begin{tabular}{ccccccc}
			\hline
			\multicolumn{7}{c}{$\mu=10^{-6}$}\\
			\hline
			n&$\3bar Q_h \bu_s -\bu_{s,h} \3bar$&$\|Q_0 \bu_s -\bu_{s,0}\|$&$\|\pi_h p_s - p_{s,h}\|$&$\3bar Q_h \bu_d -\bu_{d,h} \3bar$&$\|Q_0 \bu_d -\bu_{d,0}\|$&$\|\pi_h p_d - p_{d,h}\|$\\
			\hline
             2     & 5.7345E+05 & 5.1507E+05 & 3.2863E-01 & 5.9679E+06 & 4.9345E+06 & 1.3022E+00 \\
             4     & 1.6224E+05 & 7.3834E+04 & 3.9940E-02 & 1.9347E+06 & 8.7817E+05 & 2.1095E-01 \\
             8     & 4.2857E+04 & 9.7943E+03 & 5.4783E-03 & 5.1750E+05 & 1.1431E+05 & 4.1068E-02 \\
             16    & 1.0939E+04 & 1.2514E+03 & 8.5609E-04 & 1.3144E+05 & 1.3907E+04 & 9.7495E-03 \\
             32    & 2.7563E+03 & 1.5773E+02 & 1.6134E-04 & 3.2982E+04 & 1.6963E+03 & 2.4091E-03 \\
			\hline
			\multicolumn{7}{c}{$\mu=1$}\\
			\hline
             2     & 2.9695E+00 & 2.4334E+00 & 2.6096E+00 & 6.4562E+00 & 5.1485E+00 & 1.4295E+00 \\
             4     & 1.8525E+00 & 7.4406E-01 & 1.9843E+00 & 2.5613E+00 & 1.2120E+00 & 6.2413E-01 \\
             8     & 1.0185E+00 & 1.8042E-01 & 1.0561E+00 & 1.0740E+00 & 2.9355E-01 & 2.3380E-01 \\
             16    & 5.5392E-01 & 5.1637E-02 & 4.8939E-01 & 5.0300E-01 & 7.5433E-02 & 6.8340E-02 \\
             32    & 2.8774E-01 & 1.4193E-02 & 2.1843E-01 & 2.4584E-01 & 1.9128E-02 & 1.7520E-02 \\
			\hline
			\multicolumn{7}{c}{$\mu=10^3$}\\
			\hline
             2     & 2.8529E+00 & 2.3262E+00 & 2.6998E+03 & 2.8066E+00 & 2.4199E+00 & 1.0869E+03 \\
             4     & 1.8368E+00 & 7.3578E-01 & 1.9913E+03 & 1.7815E+00 & 9.5084E-01 & 6.1345E+02 \\
             8     & 1.0169E+00 & 1.7994E-01 & 1.0563E+03 & 9.5669E-01 & 2.7889E-01 & 2.3169E+02 \\
             16    & 5.5376E-01 & 5.1611E-02 & 4.8939E+02 & 4.8756E-01 & 7.4679E-02 & 6.7722E+01 \\
             32    & 2.8773E-01 & 1.4192E-02 & 2.1843E+02 & 2.4387E-01 & 1.9086E-02 & 1.7357E+01 \\
			\hline
		\end{tabular}
	}
\end{table}

\begin{table}
	\caption{The errors and convergence orders of Algorithm \ref{Algorithm1} in Example \ref{example 1} with $k=1$}
	\label{table1}
	\centering
	{
		\begin{tabular}{ccccccc}
			\hline
			n&$\3bar Q_h \bu_s -\bu_{s,h} \3bar$&order&$\|Q_0 \bu_s -\bu_{s,0}\|$&order&$\|\pi_h p_s - p_{s,h}\|$&order\\
			\hline
			&&&$\mu =1 $&&&\\
			\hline
            2     & 2.5649E+00 &-- & 2.0451E+00 & -- & 2.5465E+00 &-- \\
            4     & 1.8771E+00 & 0.4504  & 7.0958E-01 & 1.5271  & 1.9940E+00 & 0.3529 \\
            8     & 1.0366E+00 & 0.8567  & 1.8034E-01 & 1.9762  & 1.0567E+00 & 0.9161   \\
            16    & 5.5714E-01 & 0.8957  & 5.1662E-02 & 1.8036  & 4.8913E-01 & 1.1112 \\
            32    & 2.8820E-01 & 0.9510  & 1.4193E-02 & 1.8639  & 2.1834E-01 & 1.1636 \\
			\hline
			n&$\3bar Q_h \bu_d -\bu_{d,h} \3bar$&order&$\|Q_0 \bu_d -\bu_{d,0}\|$&order&$\|\pi_h p_d - p_{d,h}\|$&order\\
			\hline
			2      & 6.4637E+00 &-- & 5.1612E+00 & --  & 1.3990E+00 & --   \\
			4      & 2.5636E+00 & 1.3342  & 1.2163E+00 & 2.0852  & 6.0150E-01 & 1.2178    \\
			8      & 1.0741E+00 & 1.2550  & 2.9434E-01 & 2.0469  & 2.3044E-01 & 1.3842  \\
			16     & 5.0300E-01 & 1.0945  & 7.5481E-02 & 1.9633  & 6.8009E-02 & 1.7606    \\
			32     & 2.4584E-01 & 1.0328  & 1.9130E-02 & 1.9802  & 1.7491E-02 & 1.9591    \\
			\hline
		\end{tabular}
	}
\end{table}

\begin{table}
	\caption{The errors and convergence orders of Algorithm \ref{Algorithm1} in Example \ref{example 1} with $k=1$}
	\label{table2}
	\centering
	{
		\begin{tabular}{ccccccc}
			\hline
			n&$\3bar Q_h \bu_s -\bu_{s,h} \3bar$&order&$\|Q_0 \bu_s -\bu_{s,0}\|$&order&$\|\pi_h p_s - p_{s,h}\|$&order\\
			\hline
			&&&$\mu =10^3 $&&&\\
			\hline
            2     & 2.5649E+00 & --  & 2.0451E+00 & --  & 2.5465E+03 & --  \\
            4     & 1.8771E+00 & 0.4504  & 7.0958E-01 & 1.5271  & 1.9940E+03 & 0.3529  \\
            8     & 1.0366E+00 & 0.8567  & 1.8034E-01 & 1.9762  & 1.0567E+03 & 0.9161 \\
            16    & 5.5714E-01 & 0.8957  & 5.1662E-02 & 1.8036  & 4.8913E+02 & 1.1112   \\
            32    & 2.8820E-01 & 0.9510  & 1.4193E-02 & 1.8639  & 2.1834E+02 & 1.1636   \\
			\hline
			n&$\3bar Q_h \bu_d -\bu_{d,h} \3bar$&order&$\|Q_0 \bu_d -\bu_{d,0}\|$&order&$\|\pi_h p_d - p_{d,h}\|$&order\\
			\hline
			2      & 6.4637E+00 & --  & 5.1612E+00 & -- & 1.3990E+03 &--  \\
			4     & 2.5636E+00 & 1.3342  & 1.2163E+00 & 2.0852  & 6.0150E+02 & 1.2178   \\
			8      & 1.0741E+00 & 1.2550  & 2.9434E-01 & 2.0469  & 2.3044E+02 & 1.3842    \\
			16    & 5.0300E-01 & 1.0945  & 7.5481E-02 & 1.9633  & 6.8009E+01 & 1.7606  \\
			32    & 2.4584E-01 & 1.0328  & 1.9130E-02 & 1.9802  & 1.7491E+01 & 1.9591  \\
			\hline
		\end{tabular}
	}
\end{table}

\begin{table}
	\caption{The errors and convergence orders of Algorithm \ref{Algorithm1} in Example \ref{example 1} with $k=1$}
	\label{table3}
	\centering
	{
		\begin{tabular}{ccccccc}
			\hline
			n&$\3bar Q_h \bu_s -\bu_{s,h} \3bar$&order&$\|Q_0 \bu_s -\bu_{s,0}\|$&order&$\|\pi_h p_s - p_{s,h}\|$&order\\
			\hline
			&&&$\mu =10^{-6} $&&&\\
			\hline
		    2     & 2.5649E+00 & --  & 2.0451E+00 & --  & 2.5465E+00 & -- \\
	     	4     & 1.8771E+00 & 0.4504  & 7.0958E-01 & 1.5271  & 1.9940E+00 & 0.3529  \\
		    8     & 1.0366E+00 & 0.8567  & 1.8034E-01 & 1.9762  & 1.0567E+00 & 0.9161  \\
		    16    & 5.5714E-01 & 0.8957  & 5.1662E-02 & 1.8036  & 4.8913E-01 & 1.1112   \\
	    	32    & 2.8820E-01 & 0.9510  & 1.4193E-02 & 1.8639  & 2.1834E-01 & 1.1636    \\		
			\hline
			n&$\3bar Q_h \bu_d -\bu_{d,h} \3bar$&order&$\|Q_0 \bu_d -\bu_{d,0}\|$&order&$\|\pi_h p_d - p_{d,h}\|$&order\\
			\hline
			2      & 6.4637E+00 & --  & 5.1612E+00 & --  & 1.3990E+00 & --   \\
			4      & 2.5636E+00 & 1.3342  & 1.2163E+00 & 2.0852  & 6.0150E-01 & 1.2178   \\
			8    & 1.0741E+00 & 1.2550  & 2.9434E-01 & 2.0469  & 2.3044E-01 & 1.3842   \\
			16    & 5.0300E-01 & 1.0945  & 7.5481E-02 & 1.9633  & 6.8009E-02 & 1.7606   \\
			32    & 2.4584E-01 & 1.0328  & 1.9130E-02 & 1.9802  & 1.7491E-02 & 1.9591 \\
			\hline
		\end{tabular}
	}
\end{table}

\begin{table}
	\caption{The errors and convergence orders of Algorithm \ref{Algorithm1} in Example \ref{example 1} with $k=2$}
	\label{table7}
	\centering
	{
		\begin{tabular}{ccccccc}
			\hline
			n&$\3bar Q_h \bu_s -\bu_{s,h} \3bar$&order&$\|Q_0 \bu_s -\bu_{s,0}\|$&order&$\|\pi_h p_s - p_{s,h}\|$&order\\
			\hline
			&&&$ \mu=1$&&&\\
			\hline
		     2     & 1.6010E+00 & --  & 1.0421E+00 & --  & 1.7189E+00 & --  \\
	       	 4     & 5.5517E-01 & 1.5280  & 1.9436E-01 & 2.4226  & 5.1539E-01 & 1.7378  \\
		     8     & 1.6371E-01 & 1.7618  & 3.4563E-02 & 2.4914  & 1.2943E-01 & 1.9935   \\
		     16    & 4.5088E-02 & 1.8603  & 5.4381E-03 & 2.6681  & 2.8766E-02 & 2.1698   \\
		     32    & 1.1754E-02 & 1.9396  & 7.5046E-04 & 2.8573  & 6.3358E-03 & 2.1828    \\		
			\hline
			n&$\3bar Q_h \bu_d -\bu_{d,h} \3bar$&order&$\|Q_0 \bu_d -\bu_{d,0}\|$&order&$\|\pi_h p_d - p_{d,h}\|$&order\\
			\hline
			2     & 1.4878E+00 & --  & 1.0089E+00 & --  & 2.1803E-01 & --  \\
			4     & 2.7884E-01 & 2.4157  & 1.1207E-01 & 3.1704  & 3.3286E-02 & 2.7115   \\
			8     & 5.7670E-02 & 2.2736  & 1.2795E-02 & 3.1307  & 3.5287E-03 & 3.2377   \\
			16    & 1.3465E-02 & 2.0987  & 1.4954E-03 & 3.0969  & 3.1932E-04 & 3.4661  \\
			32    & 3.2858E-03 & 2.0348  & 1.7896E-04 & 3.0629  & 3.2830E-05 & 3.2819 \\
			\hline
		\end{tabular}
	}
\end{table}

\begin{table}
	\caption{The errors and convergence orders of Algorithm \ref{Algorithm1} in Example \ref{example 1} with $k=2$}
	\label{table8}
	\centering
	{
		\begin{tabular}{ccccccc}
			\hline
			n&$\3bar Q_h \bu_s -\bu_{s,h} \3bar$&order&$\|Q_0 \bu_s -\bu_{s,0}\|$&order&$\|\pi_h p_s - p_{s,h}\|$&order\\
			\hline
			&&&$ \mu=10^3$&&&\\
			\hline
            2     & 1.6010E+00 & --  & 1.0421E+00 & --  & 1.7189E+03 &--  \\
            4     & 5.5517E-01 & 1.5280  & 1.9436E-01 & 2.4226  & 5.1539E+02 & 1.7378    \\
            8     & 1.6371E-01 & 1.7618  & 3.4563E-02 & 2.4914  & 1.2943E+02 & 1.9935   \\
            16    & 4.5088E-02 & 1.8603  & 5.4381E-03 & 2.6681  & 2.8766E+01 & 2.1698  \\
            32    & 1.1754E-02 & 1.9396  & 7.5046E-04 & 2.8573  & 6.3358E+00 & 2.1828   \\
			\hline
			n&$\3bar Q_h \bu_d -\bu_{d,h} \3bar$&order&$\|Q_0 \bu_d -\bu_{d,0}\|$&order&$\|\pi_h p_d - p_{d,h}\|$&order\\
			\hline
			2      & 1.4878E+00 & --  & 1.0089E+00 & --  & 2.1803E+02 & --  \\
			4      & 2.7884E-01 & 2.4157  & 1.1207E-01 & 3.1704  &  3.3286E+01 & 2.7115 \\
			8      & 5.7670E-02 & 2.2736  & 1.2795E-02 & 3.1307  & 3.5287E+00 & 3.2377  \\
			16     & 1.3465E-02 & 2.0987  & 1.4954E-03 & 3.0969  & 3.1932E-01 & 3.4661   \\
			32     & 3.2858E-03 & 2.0348  & 1.7896E-04 & 3.0629  & 3.2830E-02 & 3.2819  \\
			\hline
		\end{tabular}
	}
\end{table}

\begin{table}
	\caption{The errors and convergence orders of Algorithm \ref{Algorithm1} in Example \ref{example 1} with $k=2$}
	\label{table9}
	\centering
	{
		\begin{tabular}{ccccccc}
			\hline
			n&$\3bar Q_h \bu_s -\bu_{s,h} \3bar$&order&$\|Q_0 \bu_s -\bu_{s,0}\|$&order&$\|\pi_h p_s - p_{s,h}\|$&order\\
			\hline
			&&&$ \mu=10^{-6}$&&&\\
			\hline
		    2     & 1.6010E+00 & --  & 1.0421E+00 & --  & 1.7189E-06 & --   \\
		    4     & 5.5517E-01 & 1.5280  & 1.9436E-01 & 2.4226  & 5.1539E-07 & 1.7378    \\
		    8     & 1.6371E-01 & 1.7618  & 3.4563E-02 & 2.4914  & 1.2943E-07 & 1.9935   \\
	     	16    & 4.5088E-02 & 1.8603  & 5.4381E-03 & 2.6681  & 2.8766E-08 & 2.1698   \\
	    	32    & 1.1754E-02 & 1.9396  & 7.5046E-04 & 2.8573  & 6.3358E-09 & 2.1828   \\
		
			\hline
			n&$\3bar Q_h \bu_d -\bu_{d,h} \3bar$&order&$\|Q_0 \bu_d -\bu_{d,0}\|$&order&$\|\pi_h p_d - p_{d,h}\|$&order\\
			\hline
			2      & 1.4878E+00 & --  & 1.0089E+00 & --  & 2.1803E-07 & -- \\
			4     & 2.7884E-01 & 2.4157  & 1.1207E-01 & 3.1704  & 3.3286E-08 & 2.7115  \\
			8     & 5.7670E-02 & 2.2736  & 1.2795E-02 & 3.1307  & 3.5287E-09 & 3.2377  \\
			16    & 1.3465E-02 & 2.0987  & 1.4954E-03 & 3.0969  & 3.1932E-10 & 3.4661   \\
			32     & 3.2858E-03 & 2.0348  & 1.7896E-04 & 3.0629  & 3.2830E-11 & 3.2819  \\
			\hline
		\end{tabular}
	}
\end{table}

\begin{figure*}[t!]
	\centering
	\begin{subfigure}[t]{0.33\linewidth}
		\centering
		\includegraphics[width=1\linewidth]{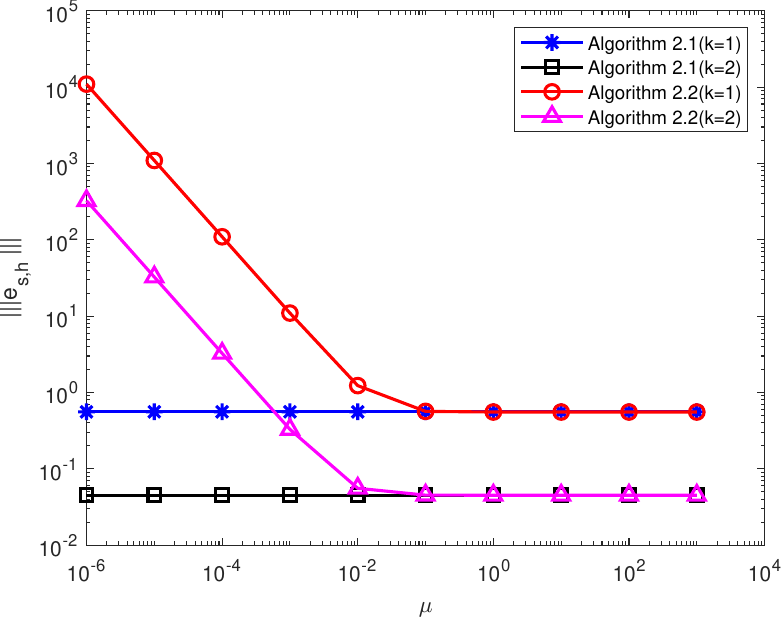}
	\end{subfigure}%
	\begin{subfigure}[t]{0.33\linewidth}
		\centering
		\includegraphics[width=1\linewidth]{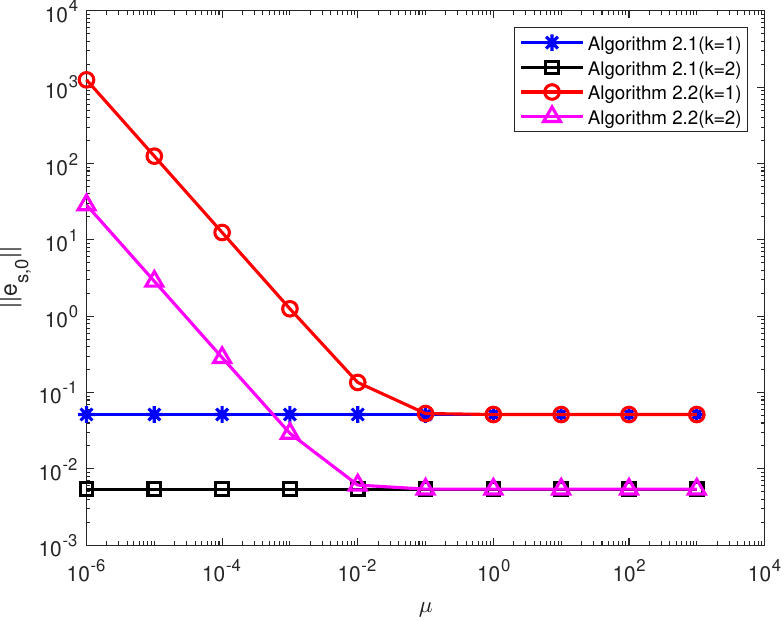}
	\end{subfigure}
	\begin{subfigure}[t]{0.33\linewidth}
		\centering
		\includegraphics[width=1\linewidth]{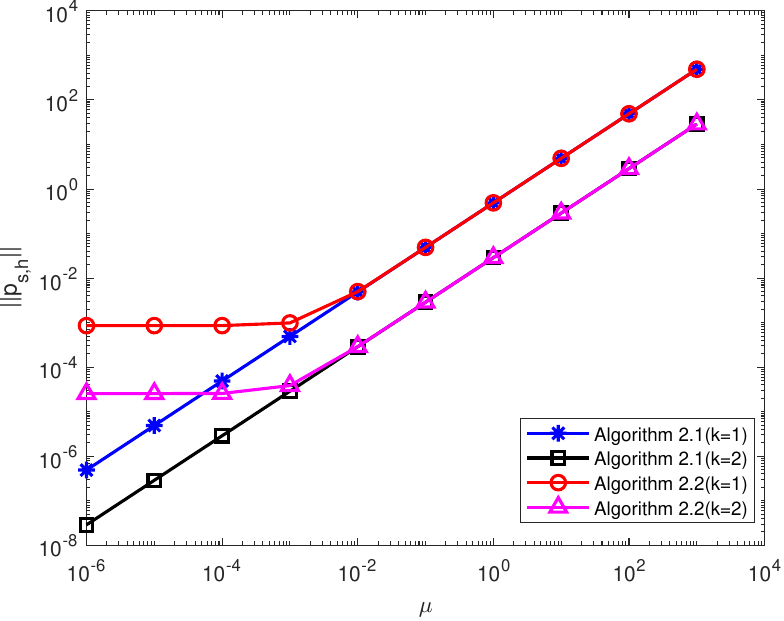}
	\end{subfigure}
	\caption{Comparison of errors between Algorithms \ref{Algorithm1} and \ref{Algorithm2} in the Stokes domain with different $\mu$}
	\label{figure6.2}
\end{figure*}

\begin{figure*}[t!]
	\centering
	\begin{subfigure}[t]{0.33\linewidth}
		\centering
		\includegraphics[width=1.0\linewidth]{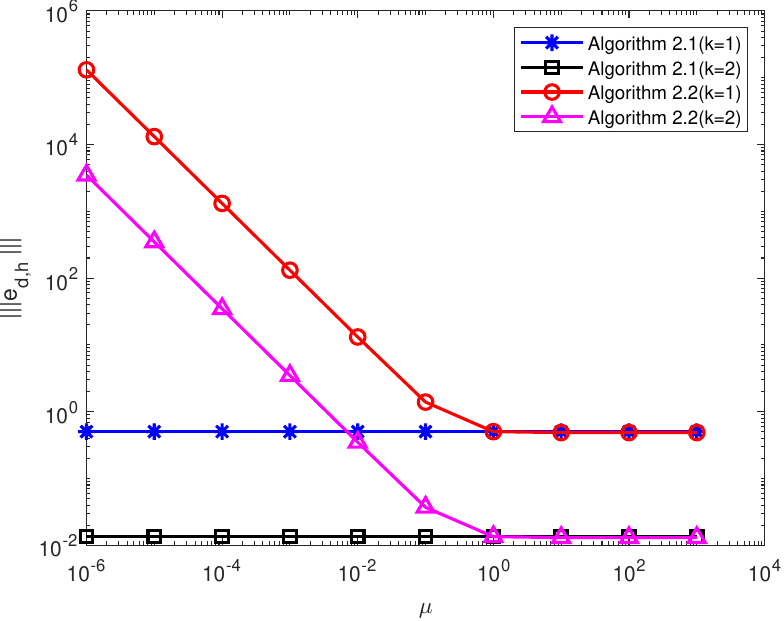}
	\end{subfigure}%
	\begin{subfigure}[t]{0.33\linewidth}
		\centering
		\includegraphics[width=1.0\linewidth]{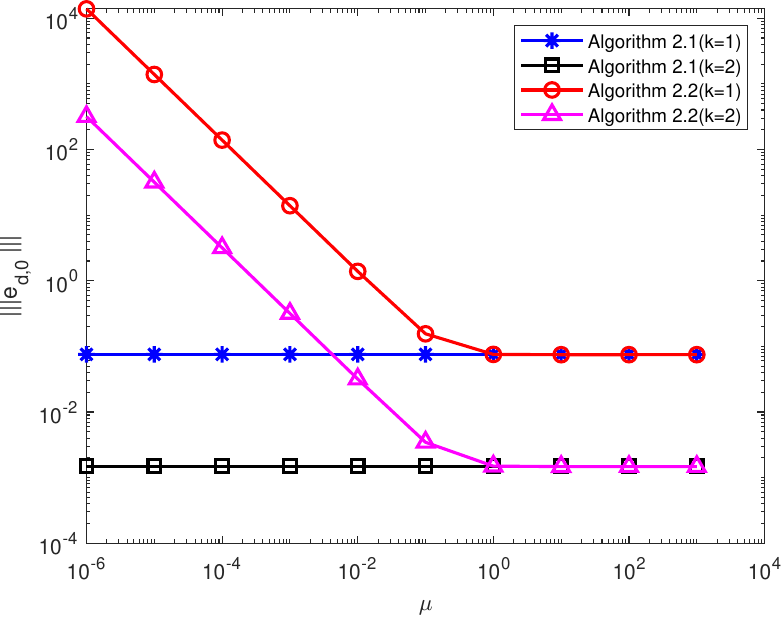}
	\end{subfigure}
	\begin{subfigure}[t]{0.33\linewidth}
		\centering
		\includegraphics[width=1.0\linewidth]{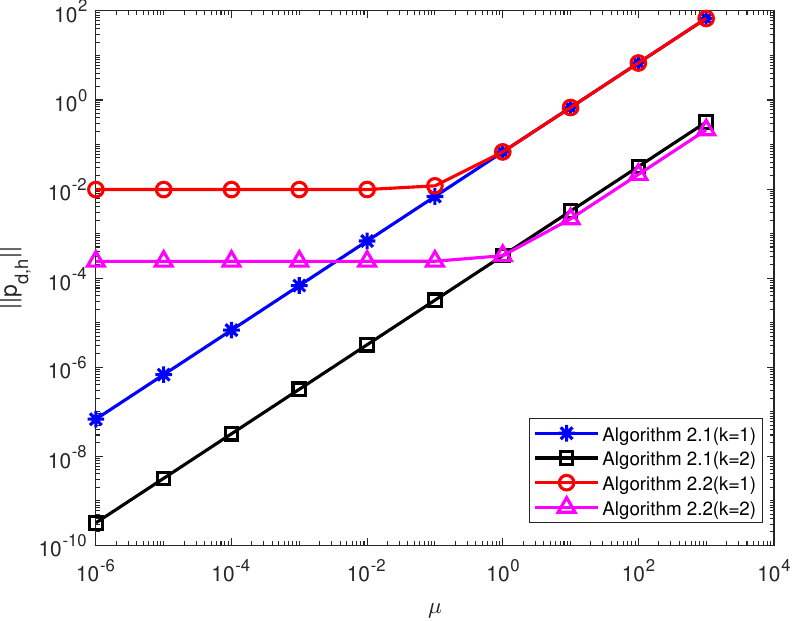}
	\end{subfigure}
	\caption{Comparison of errors between Algorithms \ref{Algorithm1} and \ref{Algorithm2} in the Darcy domain with different $\mu$}
	\label{figure6.3}
\end{figure*}

\begin{table}
\caption{{\color{black}The errors of Algorithm \ref{Algorithm1} in Example \ref{example 2} with different $k$}}
\label{table13}
\centering
	\resizebox{0.9\textwidth}{!}
{
	\begin{tabular}{ccccccc}
		\hline
		\multicolumn{7}{c}{$k=1$}\\
		\hline
		n&$\3bar Q_h \bu_s -\bu_{s,h} \3bar$&$\|Q_0 \bu_s -\bu_{s,0}\|$&$\|\pi_h p_s - p_{s,h}\|$&$\3bar Q_h \bu_d -\bu_{d,h} \3bar$&$\|Q_0 \bu_d -\bu_{d,0}\|$&$\|\pi_h p_d - p_{d,h}\|$\\
		\hline
         2     & 0.0000E+00 & 0.0000E+00 & 0.0000E+00 & 0.0000E+00 & 0.0000E+00 & 0.0000E+00 \\
         4     & 0.0000E+00 & 0.0000E+00 & 0.0000E+00 & 1.0000E-15 & 1.0000E-15 & 0.0000E+00 \\
         8     & 1.0000E-15 & 0.0000E+00 & 1.0000E-15 & 1.0000E-15 & 1.0000E-15 & 0.0000E+00 \\
         16    & 3.0000E-15 & 0.0000E+00 & 3.0000E-15 & 5.0000E-15 & 3.0000E-15 & 1.0000E-15 \\
         32    & 3.0000E-14 & 2.0000E-15 & 2.9000E-14 & 2.2000E-14 & 8.0000E-15 & 1.1000E-14 \\
		\hline
				\multicolumn{7}{c}{$k=2$}\\
		\hline
		    2     & 1.0000E-15 & 0.0000E+00 & 0.0000E+00 & 3.0000E-15 & 4.0000E-15 & 0.0000E+00 \\
		4     & 1.0000E-15 & 0.0000E+00 & 1.0000E-15 & 7.0000E-15 & 9.0000E-15 & 1.0000E-15 \\
		8     & 1.1000E-14 & 0.0000E+00 & 1.1000E-14 & 2.4000E-14 & 2.0000E-14 & 2.0000E-15 \\
		16    & 1.8000E-14 & 0.0000E+00 & 1.8000E-14 & 6.9000E-14 & 4.2000E-14 & 5.0000E-15 \\
		32    & 7.8000E-14 & 1.0000E-15 & 7.6000E-14 & 1.7900E-13 & 8.9000E-14 & 2.1000E-14 \\
		\hline
	\end{tabular}
}
\end{table}


\begin{table}
	\caption{The errors and convergence orders of Algorithm \ref{Algorithm2} in Example \ref{example 2} with $k=1$}
	\label{table15}
	\centering
	{
		\begin{tabular}{ccccccc}
			\hline
			n&$\3bar Q_h \bu_s -\bu_{s,h} \3bar$&order&$\|Q_0 \bu_s -\bu_{s,0}\|$&order&$\|\pi_h p_s - p_{s,h}\|$&order\\
			\hline
             2     & 6.8112E-03 & --  & 1.4570E-03 & --  & 1.8508E-03 & --   \\
             4     & 1.9684E-03 & 1.7909  & 2.1347E-04 & 2.7709  & 3.0793E-04 & 2.5875    \\
             8     & 5.2126E-04 & 1.9170  & 2.8576E-05 & 2.9011  & 5.9622E-05 & 2.3687    \\
             16    & 1.3340E-04 & 1.9662  & 3.6743E-06 & 2.9593  & 1.3053E-05 & 2.1915    \\
             32    & 3.3686E-05 & 1.9856  & 4.6468E-07 & 2.9831  & 3.0983E-06 & 2.0748   \\
			\hline
			n&$\3bar Q_h \bu_d -\bu_{d,h} \3bar$&order&$\|Q_0 \bu_d -\bu_{d,0}\|$&order&$\|\pi_h p_d - p_{d,h}\|$&order\\
			\hline
            2      & 4.8209E-02 & --  & 1.4021E-02 &--  & 6.3653E-03 & -- \\
            4    & 1.2943E-02 & 1.8972  & 1.7870E-03 & 2.9720  & 1.5962E-03 & 1.9956 \\
            8    & 3.2903E-03 & 1.9758  & 2.1625E-04 & 3.0467  & 3.9769E-04 & 2.0050  \\
            16   & 8.2676E-04 & 1.9927  & 2.6311E-05 & 3.0390  & 9.9265E-05 & 2.0023 \\
            32    & 2.0701E-04 & 1.9978  & 3.2373E-06 & 3.0228  & 2.4804E-05 & 2.0007  \\
			\hline
		\end{tabular}
	}
\end{table}
\begin{table}
	\caption{The errors and convergence orders of Algorithm \ref{Algorithm2} in Example \ref{example 2} with $k=2$}
	\label{table16}
	\centering
	{
		\begin{tabular}{ccccccc}
			\hline
			n&$\3bar Q_h \bu_s -\bu_{s,h} \3bar$&order&$\|Q_0 \bu_s -\bu_{s,0}\|$&order&$\|\pi_h p_s - p_{s,h}\|$&order\\
			\hline
            2     & 1.5290E-03 &--  & 2.6384E-04 & --  & 3.2078E-04 & --   \\
            4     & 2.0851E-04 & 2.8744  & 1.7960E-05 & 3.8768  & 3.2147E-05 & 3.3188    \\
            8     & 2.6774E-05 & 2.9612  & 1.1464E-06 & 3.9696  & 3.3041E-06 & 3.2824   \\
            16    & 3.3775E-06 & 2.9868  & 7.1905E-08 & 3.9949  & 3.5397E-07 & 3.2226   \\
            32    & 4.2366E-07 & 2.9950  & 4.4910E-09 & 4.0010  & 3.9705E-08 & 3.1562    \\
			\hline
			n&$\3bar Q_h \bu_d -\bu_{d,h} \3bar$&order&$\|Q_0 \bu_d -\bu_{d,0}\|$&order&$\|\pi_h p_d - p_{d,h}\|$&order\\
			\hline
            2     & 7.9615E-03 &-- & 1.8062E-03 & -- & 8.2251E-04 & --\\
            4     & 1.0386E-03 & 2.9384  & 1.1530E-04 & 3.9695  & 1.0318E-04 & 2.9949  \\
            8     & 1.3118E-04 & 2.9850  & 7.1442E-06 & 4.0125  & 1.2894E-05 & 3.0003   \\
            16    & 1.6440E-05 & 2.9963  & 4.4209E-07 & 4.0144  & 1.6113E-06 & 3.0004   \\
            32    & 2.0564E-06 & 2.9990  & 2.7450E-08 & 4.0094  & 2.0140E-07 & 3.0002  \\
			\hline
		\end{tabular}
	}
\end{table}

\section{Conclusion}
In this paper, we propose the pressure-robust weak Galerkin finite element scheme to solve the Stokes-Darcy problem by constructing the divergence-free velocity reconstruction operator. The proposed WG scheme can improve the velocity function and pressure function simultaneously. We prove that the error of the velocity function are independent of the pressure function and viscosity coefficient. And the numerical solutions converge to the exact solution at the optimal orders in the $H^1$ norm and $L^2$ norm. The numerical results agree with the theoretical analysis. This observation demonstrates that the proposed WG scheme is pressure-robust and efficient to solve the Stokes-Darcy problem.

\vspace{\baselineskip}
\noindent {\bf Acknowledgments.} This research was supported by the National Natural Science Foundation of China (grant No. 11901015, 12271208, 12001232, 12201246, 22341302), the National Key Research and Development Program of China (grant No. 2020YFA0713602, 2023YFA1008803), Fundamental Research Funds for the Central Universities housed at Jilin University (grant No. 93Z172023Z05), and the Key Laboratory of Symbolic Computation and Knowledge Engineering of Ministry of Education of China housed at Jilin University.


\begin{thebibliography}{10}


\bibitem{bookfiniteelementmethod1} S.~C. Brenner and L.~R. Scott, \emph{The mathematical theory of finite element methods},  Springer-Verlag, 2002.
\bibitem{bookfiniteelementmethod2} F.~Brezzi and M.~Fortin, \emph{ Mixed and hybrid finite element methods},
 Springer-Verlag, 1991.
\bibitem{WGStokesDrcyWang} W.~Chen, F.~Wang, and Y.~Wang,  Weak {G}alerkin method for the coupled {D}arcy-{S}tokes flow, \emph{IMA J. Numer. Anal.}, \textbf{36}:2(2016), 897--921.
\bibitem{bookfiniteelementmethod3} M.~Crouzeix and P.-A. Raviart,  Conforming and nonconforming finite
	element methods for solving the stationary {S}tokes equations. {I},  \emph{ Rev. Fran\c{c}aise Automat. Informat. Recherche Op\'{e}rationnelle S\'{e}r. Rouge}, \textbf{7}:3(1973), 33--75.
\bibitem{MFEMSD2} M.~Discacciati, E.~Miglio, and A.~Quarteroni, Mathematical and
	numerical models for coupling surface and groundwater flows, \emph{Appl. Numer. Math.}, \textbf{43}:1-2(2002), 57--74.
\bibitem{ern_finite_2021} A.~Ern and J.-L. Guermond, \emph{Finite {Elements} {I}: {Approximation}
	and {Interpolation}}, Springer, 2021.	
\bibitem{Graddiv1988} L.~P. Franca and T.~J.~R. Hughes, Two classes of mixed finite
	element methods, \emph{Comput. Methods Appl. Mech. Engrg.}, \textbf{69}:1(1988), 89--129.	
\bibitem{MFEMSD1}G.~N. Gatica, R.~Oyarz\'{u}a, and F. -J. Sayas, Analysis of
	fully-mixed finite element methods for the {S}tokes-{D}arcy coupled problem, \emph{Math. Comp.}, \textbf{80}:276(2011), 1911--1948.	
\bibitem{DGSD2} G.~N. Gatica and F.~A. Sequeira, Analysis of the {HDG} method for
	the {S}tokes-{D}arcy coupling, \emph{Numer. Methods Partial Differential Equations}, \textbf{33}(2017), 885--917.	
\bibitem{bookfiniteelementmethod5} V.~Giraud and P.~Raviart, \emph{Finite Element Methods for the Navier-Stokes Equations, Theory and Algorithms}, Springer, 1986.
\bibitem{girault2012finite} V.~Girault and P.-A. Raviart, \emph{Finite element methods for
	Navier-Stokes equations: theory and algorithms}, Springer, 2012.
\bibitem{bookfiniteelementmethod4} M.~D. Gunzburger, \emph{Finite element methods for viscous incompressible
	flows}, Acdemic Press, Inc., 1989.
\bibitem{DGSD3} I.~Igreja and A.~F.~D. Loula, A stabilized hybrid mixed {DGFEM}
	naturally coupling {S}tokes-{D}arcy flows, \emph{Comput. Methods Appl. Mech.
		Engrg.}, \textbf{339}(2018), 739--768.
\bibitem{PressureRobust2017} V.~John, A.~Linke, C.~Merdon, M.~Neilan, and L.~G. Rebholz, On the
	divergence constraint in mixed finite element methods for incompressible
	flows, \emph{SIAM Rev.}, \textbf{59}:3(2017), 492--544.
\bibitem{Graddiv2009} W.~Layton, C.~C. Manica, M.~Neda, M.~Olshanskii, and L.~G. Rebholz,
	On the accuracy of the rotation form in simulations of the {N}avier-{S}tokes
	equations, \emph{J. Comput. Phys.}, \textbf{228}:9(2009), 3433--3447.
\bibitem{MFEMSD3} W.~J. Layton, F.~Schieweck, and I.~Yotov, Coupling fluid flow with
	porous media flow, \emph{SIAM J. Numer. Anal.}, \textbf{40}:6(2002), 2195--2218.	
\bibitem{WGSD1} R.~Li, Y.~Gao, J.~Li, and Z.~Chen, A weak {G}alerkin finite element
	method for a coupled {S}tokes-{D}arcy problem on general meshes, \emph{J. Comput.
		Appl. Math.}, \textbf{334} (2018), 111--127.	
\bibitem{WGSD2} R.~Li, J.~Li, X.~Liu, and Z.~Chen, A weak {G}alerkin finite element
	method for a coupled {S}tokes-{D}arcy problem, \emph{Numer. Methods Partial
		Differential Equations}, \textbf{33}:4(2017), 1352--1373.
\bibitem{Reconstruction1} A.~Linke, A divergence-free velocity reconstruction for
	incompressible flows, \emph{C. R. Math. Acad. Sci. Paris}, \textbf{350}:17-18(2012), 837--840.
\bibitem{Reconstruction2}, A.~Linke, On the role of the
	{H}elmholtz decomposition in mixed methods for incompressible flows and a new
	variational crime, \emph{Comput. Methods Appl. Mech. Engrg.}, \textbf{268}(2014), 782--800.
\bibitem{Reconstruction3}A.~Linke, C.~Merdon, and W.~Wollner, Optimal $L^2$ velocity error
	estimate for a modified pressure-robust crouzeix--raviart stokes element, \emph{IMA Journal of Numerical Analysis}, \textbf{37}:1(2017), 354--374.
\bibitem{MFEMPR1} D.~Lv and H.~Rui, A pressure-robust mixed finite element method for
	the coupled {S}tokes-{D}arcy problem, \emph{J. Comput. Appl. Math.}, \textbf{436}(2024), 115444.
\bibitem{PressureRobust2020} L.~Mu, Pressure robust weak {G}alerkin finite element methods for
	{S}tokes problems, \emph{SIAM J. Sci. Comput.}, \textbf{42}:3 (2020), B608--B629.
\bibitem{PressureRobustCDG2021}L.~Mu, X.~Ye, and S.~Zhang, Development of pressure-robust
	discontinuous {G}alerkin finite element methods for the {S}tokes problem, \emph{J.
		Sci. Comput.}, \textbf{89}:1 (2021), 26.
\bibitem{PressureRobust2021}
L.~Mu, X.~Ye, and S.~Zhang, A stabilizer-free, pressure-robust, and superconvergence weak {G}alerkin finite element method
	for the {S}tokes equations on polytopal mesh, \emph{SIAM J. Sci. Comput.}, \textbf{43}:4(2021), A2614--A2637.
\bibitem{Graddiv2009-2} M.~Olshanskii, G.~Lube, T.~Heister, and J.~L\"{o}we, Grad-div
	stabilization and subgrid pressure models for the incompressible
	{N}avier-{S}tokes equations, \emph{Comput. Methods Appl. Mech. Engrg.}, \textbf{198}:49-52(2009), 3975--3988.
\bibitem{Graddiv2002} M.~A. Olshanskii, A low order {G}alerkin finite element method for
	the {N}avier-{S}tokes equations of steady incompressible flow: a
	stabilization issue and iterative methods, \emph{Comput. Methods Appl. Mech.
		Engrg.}, \textbf{191}:47-48(2002), 5515--5536.
\bibitem{Graddiv2004} M.~A. Olshanskii and A.~Reusken, Grad-div stabilization for {S}tokes
	equations, \emph{Math. Comp.}, \textbf{73}:248(2004), 1699--1718.
\bibitem{WGStokesDrcyPeng2} H.~Peng, Q.~Zhai, R.~Zhang, and S.~Zhang, Weak {G}alerkin and
	continuous {G}alerkin coupled finite element methods for the {S}tokes-{D}arcy
	interface problem, \emph{Commun. Comput. Phys.}, \textbf{28}:3(2020), pp.~1147--1175.
\bibitem{WGStokesDrcyPeng1}
H.~Peng, Q.~Zhai, R.~Zhang, and S.~Zhang, A weak {G}alerkin-mixed finite element method for the {S}tokes-{D}arcy problem,
    \emph{Sci. China Math.}, \textbf{64}:10 (2021), 2357--2380.
\bibitem{DGSD4} B.~Rivi\`ere and I.~Yotov, Locally conservative coupling of {S}tokes
	and {D}arcy flows, \emph{SIAM J. Numer. Anal.}, \textbf{42}:5(2005), 1959--1977.
\bibitem{PressureRobust2021.2}
B.~Wang and L.~Mu, Viscosity robust weak {G}alerkin finite element
	methods for {S}tokes problems, \emph{Electron. Res. Arch.}, \textbf{29}:1(2021), 1881--1895.
\bibitem{VEMSD} G.~Wang, F.~Wang, L.~Chen, and Y.~He, A divergence free weak virtual
	element method for the {S}tokes-{D}arcy problem on general meshes, \emph{Comput.
		Methods Appl. Mech. Engrg.}, \textbf{344}: (2019), pp.~998--1020.	
\bibitem{wang2013weak} J.~Wang and X.~Ye, A weak galerkin finite element method for
	second-order elliptic problems, \emph{J. Comput. Appl. Math.}, \textbf{241}(2013), pp.~103--115.
\bibitem{WGStokes1} J.~Wang and X.~Ye, A weak {G}alerkin
	finite element method for the stokes equations, \emph{Adv. Comput. Math.}, \textbf{42}:1(2016), pp.~155--174.
\bibitem{DGSD} J.~Wen, J.~Su, Y.~He, and H.~Chen,  A discontinuous {G}alerkin method
	for the coupled stokes and {D}arcy problem, \emph{J. Sci. Comput.}, \textbf{85}:2(2020), 26.	
\bibitem{FVMSD} C.~M. Xie, Y.~Luo, and M.~F. Feng, Analysis of a unified stabilized
	finite volume method for the {D}arcy-{S}tokes problem, \emph{Math. Numer. Sin.}, \textbf{33}:2(2011), 133--144.		
\end{thebibliography}
\end{document}